\documentclass[a4paper, 11pt, final, BCOR=5mm, pdftex, pagesize]{scrartcl} 
\usepackage[utf8]{inputenc} 
\usepackage{fancyhdr}
\usepackage{csquotes} 
\usepackage{verbatim}
\usepackage{underscore}
\usepackage{enumitem}
\usepackage{color}
\usepackage{fixltx2e} 
\usepackage{lmodern} 
\usepackage[T1]{fontenc} 
\usepackage{setspace}

\providecommand{\keywords}[1]
{
  \small	
  \textbf{\textit{Keywords---}} #1
}

\usepackage{amsmath}
\usepackage{amsfonts}
\usepackage{amssymb}
\usepackage{mathrsfs}
\usepackage{mathtools}
\usepackage{dsfont} 
\usepackage{bigints} 
\usepackage{amsthm}
\usepackage{accents}

\usepackage{float} 
\usepackage[margin=0.8cm]{caption} 
\usepackage{subcaption}
\usepackage{graphicx}
\usepackage{pgfplots}
\pgfplotsset{compat=1.8}
\usepackage[draft=false,babel,tracking=true,kerning=true,spacing=true]{microtype} 
\usepackage{tabularx} 

\definecolor{CMblue}{RGB}{0, 153, 153}
\definecolor{CMgreen}{RGB}{76, 153, 0}
\definecolor{CMorange}{RGB}{237, 125, 49}
\definecolor{CMlightblue}{RGB}{204, 255, 255}

\usepackage{chngcntr}
\usepackage[section]{placeins} 
\numberwithin{equation}{section}

\usepackage[noadjust]{cite}
\usepackage{hyperref}

\usepackage{tikz}
\usetikzlibrary{shapes.geometric, patterns}
\pgfplotsset{compat=1.8}
\usepackage{setspace,tabularx} 
\usepackage[draft=false,babel,tracking=true,kerning=true,spacing=true]{microtype} 

\DeclareMathOperator*{\argmax}{arg\,max}

\newcommand{\R}{\mathbb{R}}
\newcommand{\N}{\mathbb{N}}

\newcommand{\E}{\ensuremath\mathbb{E}}
\newcommand{\1}{\textup{$\mathds{1}$}}
\newcommand{\Pro}{\ensuremath\mathbb{P}}

\newcommand{\Var}{\ensuremath\mathrm{Var}}
\newcommand{\Cov}{\ensuremath\mathrm{Cov}}

\newcommand{\inv}{\ensuremath\mathrm{inv}}

\renewcommand{\leq}{\leqslant}
\renewcommand{\geq}{\geqslant}
\renewcommand{\bar}{\overline}
\renewcommand{\tilde}{\widetilde}

\newtheorem{The}{Theorem}[section]
\newtheorem*{The*}{Theorem}
\newtheorem{Ass}{Assumption}
\newtheorem{Exp}{Example}

\newtheorem{Lem}[The]{Lemma}
\newtheorem{Cor}[The]{Corollary}
\newtheorem{Rmk}[The]{Remark}
\newtheorem{Prop}[The]{Proposition}

\begin{document}

\title{Parametric change point detection with random occurrence of the change point}
\author{Cassandra Milbradt \footnote{Humboldt-Universität zu Berlin, Germany. E-mail: cassandra.milbradt@gmail.com}}
\date{}

\maketitle

\vspace{-1cm}

\begin{abstract}
    We are concerned with the problem of detecting a single change point in the model parameters of time series data generated from an exponential family. In contrast to the existing literature, we allow that the true location of the change point is itself random, possibly depending on the data. Under the alternative, we study the case when the size of the change point converges to zero while the sample size goes to infinity. Moreover, we concentrate on change points in the ``middle of the data'', i.e., we assume that the change point fraction (the location of the change point relative to the sample size) converges weakly to a random variable $\lambda^*$ which takes its values almost surely in a closed subset of $(0,1).$ We show that the known statistical results from the literature also transfer to this setting. We substantiate our theoretical results with a simulation study.
\end{abstract}


\section{Introduction}\label{RCP:sec:introduction}

Detecting structural changes in the parameters of time series data is of great interest from both econometric and statistical perspectives. Traditionally, one is faced with the question of whether the underlying time series data contain one or more change points. While some literature (cf. e.g. \cite{CH97,BJV17,AHHR09}) has also investigated the detection of multiple change points, in this work we focus only on so-called ``at most one change point'' (AMOC) models. Assuming the location of the change point is known, one can interpret the question of deciding whether or not the data contain a change point as a two-sample test problem. However, the location of a change point is typically unknown. Many works on change point models already provide statistical tests to answer this question for unknown change points. While many authors build their tests assuming that the change point occurs only in a single model parameter (typically in its mean, cf. e.g. \cite{JWY18} or in its variance, cf. e.g. \cite{AHHR09, S09}), Horváth \cite{H93}, Gombay and Horváth \cite{GH94, GH96a, GH96b}, and Csörg\H{o} and Horváth  \cite{CH97} provide likelihood ratio-based tests that check for a simultaneous change in the parameters of quite general parametric distributions including exponential families, see Csörg\H{o} and Horváth \cite{CH97} for an overview. Under a long-span asymptotic scheme, in which the time span of data is assumed to go to infinity, the existing literature provides theory for the estimation of a fractional change point (the location of the change point relative to the sample size), including the consistency, the rate of convergence, and the limiting distribution (cf. e.g. \cite{CH97, JWY18, AHHR09}). Although in most of the literature, time series data are considered, some authors study the detection of change points in the drift and/or volatility process of continuous-time diffusions or more general It\^{o}-semimartingales assuming a continuous record or a discrete-time record with mesh size converging to zero over a finite time span is available (cf. e.g. \cite{IY12, JWY18, BJV17}). However, if the size of the model parameters relates appropriately with the samples size, one might approximate the time series data by a continuous-time model and hence might be able to connect the findings for time series data with the theory developed for continuous-time processes. For example, if the time series data are normally distributed and the mean is of order $n^{-1/2}$ while the volatility is of order $1$ ($n$ denotes the sample size), for large $n$, the $n^{-1/2}$-scaled partial sum can be approximated by a diffusion process. Because of the different scaling in $n$ in the mean and volatility, a change point in the mean is typically much harder to detect than in the volatility. Even more, when studying the detection of change points in the mean, the existing literature reveals that the consistency of an estimator for the change point can only be obtained if the size of the change point is of larger order than $n^{-1/2}.$ But then, the approximation of the time series data by a continuous-time model fails since the size of the change point explodes as $n\rightarrow \infty.$ For this reason, \cite{JWY18} studied the asymptotic properties of the change point estimator in the crucial case when the shift in the mean is of order $n^{-1/2}.$ While in this work, we mainly focus on parametric models, the recent literature also provides tools for the detection of structural changes in non-parametric models such as in the volatility process of an It\^{o}-semimartingale (cf. e.g. \cite{BJV17}), or the mean or location parameter of time series data (cf. e.g. \cite{CH97}). \par
Calibration of mathematical models is one of the main concerns from a practitioners' point of view. It is well known that change points are present in high-frequency financial data. In the referenced literature the location of the change point is unknown but deterministic. However, if the change point is caused by endogenous effects, the dependence on the underlying data must be considered. The integrated European intraday electricity market ``Single Intraday Coupling'' (SIDC) is a real-world example in which change points are endogenously caused. In this market, multiple national limit order books are coupled, i.e., summarized in a single shared order book such that market orders are allowed to be matched with standing volumes of the domestic and foreign limit order books. However, the coupling of multiple markets is only maintained as long as transmission capacities are available. In contrast, if the transmission capacities are fully occupied, market orders can only be matched with standing volumes of the same origin. The switch between these two regimes typically leads to structural changes in the trading behavior. In Milbradt \cite{M22}, we construct cross-border market dynamics including prices, standing volumes at the best bid and ask prices, and capacities from the underlying net order flow process and the total available transmission capacities. The time of a regime switch is then modeled by a stopping time depending on the net order flow and on the total available capacities. While the order flow is publicly available, the transmission capacities are harder to obtain and therefore often unknown. Hence, in order to calibrate the model to high-frequency data, the time of a regime switch, that depends on the observed data, must be estimated.\par 

Models that have been studied in the literature, in which the location of a change point is itself random are, for example, so-called Markov switching models. In these models, the type of the regime depends on an unobserved Markov process which is independent of the data (cf. e.g. \cite{CW07, EKS08}). Despite this, to the best of our knowledge, the existing literature on random occurring change points is rather limited. In our work, we extend the statistical results in Csörg\H{o} and Horváth \cite{CH97} to randomly occurring change points, possibly depending on the data. Throughout, we assume that the data points are independent and only study the case when the size of the change point converges to zero while the sample size $n$ goes to infinity. From a statistical point of view, this case describes the crucial setting as it answers the question of which minimum size of a change point is detectable, i.e., if the null hypothesis $H_0$ is ``no change point'' versus the alternative $H_1$ is ``there is one change point'', then we discuss the problem of distinguishing between $H_0$ and $H_1.$ Moreover, under the alternative $H_1,$ if $k^*_n \in \{1,\cdots, n-1\}$ denotes the true but random location of the change point, we concentrate on change points in the ``middle of the data'', i.e., we assume that the change point fraction $k^*_n/n \Rightarrow \lambda^*,$ where $\lambda^*$ is a random variable taking values in a closed subset of $(0,1)$ with probability one. \par 
Our model should be understood as a first proof of concept for the extension of the very general change point theory in \cite{CH97} to randomly occurring change points. We show that the statistical properties of the test statistic as well as of the estimator for the location of the change point transfer from the deterministic setting considered in \cite{CH97} to randomly occurring change points. While this might be clear under the null hypothesis $H_0$, this is not obvious under the alternative. In particular, our work shows that the theory in \cite{CH97} can also be applied in the model framework introduced in \cite{M22} in which the location of a regime switch depends on the underlying net order flow process. \par 
To extend the results in \cite{CH97}, the main difficulty is to show that the limit result for the test statistic under the alternative holds true uniformly for all possible values of the location of the change point. Therefore, we introduce an alternative test statistic depending on two time parameters (i.e. on the true and estimated location of the change point) and show that if this test statistic is scaled appropriately, it converges weakly in the Skorokhod topology to a Gaussian process with two time parameters. The hard part of the proof turns out to correctly identify the finite-dimensional distributions of the limit process. This can be nicely simplified by an application of the fourth moment theorem (cf. Theorem 1 in \cite{NP05}) since we concentrate on normally distributed data when studying the asymptotics under the alternative. After establishing the limit theorem of the test statistic under the alternative, it is indeed straight-forward to prove the known results in \cite{CH97} also for randomly occurring change points. We provide empirical support for our theoretical results through a detailed simulation study. Moreover, in this simulation study we also discuss two important generalizations of our model: weakly dependent observations and non-parametric change point detection in the volatility process of an It\^o-semimartingale. It turns out, at least empirically, that change point detection works for these cases as well, even if the location of the change point depends on the data. \newline

\textbf{Structure of this paper:} In Section $2$, we introduce the model framework and the test statistic based on the so-called maximally selected log-likelihood ratio. Since under the null hypothesis, no change point occurs, in Section $3$, we repeat the results for the asymptotics of the considered test statistic under the null hypothesis in \cite{CH97}. Under the alternative, our more general setting of a change point with random location becomes important. Therefore, in Section $4$, we present a new test statistic depending on the location of a change point and derive its limit distribution relative to the location of a change point (cf. Theorem \ref{RCP:res:limitZn1}). To simplify the proof, we assume in Section $4$ that the observations are normally distributed. Moreover, we introduce an estimator for the fractional change point and establish its consistency, the convergence rate, and the limit distribution. The latter is also stated in a distribution-free version, which allows to build confidence intervals for the true location of the change point based on the data. We finish this paper by a detailed simulation study in Section $5$.\newline

\textbf{Notation:} In the following, for each $x\in \R^d,$ $d \geq 1,$ let us denote by $\|x\|:= (x^Tx)^{1/2}$ the euclidean norm in $\R^d.$ Moreover, we write $\Pro[A, B] := \Pro[A \cap B]$ for $A,B \in \mathcal{F}$ and a probability space $(\Omega, \mathcal{F}, \Pro)$.

\section{Setup}

Throughout, we assume that all random variables are defined on some common probability space $(\Omega, \mathcal{F}, \Pro).$ Let $X_1, \cdots, X_n$ be independent observations in $\R^m$ which have densities $f_{X_j},$ $j = 1, \cdots, n,$ with respect to some $\sigma$-finite measure $\nu$ being element of the exponential family, i.e.,
\begin{equation}\label{RCP:ass:density}
f_{X_j}(x) = f(x; \theta_j) = \exp\left(\theta_j^T T(x)+S(x)- A(\theta_j)\right)\1_{\{x \in C\}},
\end{equation}
where $x = (x_1, \cdots, x_m)^T,$ $\theta_j = (\theta_{j,1}, \cdots, \theta_{j,d})^T \in \Theta \subset \R^d,$ $S, T_1, \cdots, T_d: (\Omega, \mathcal{F}) \rightarrow (\R, \mathcal{B}(\R))$ are measurable functions with $T=(T_1,\cdots, T_d)^T,$ $A: \R^d \rightarrow \R,$ and $C \subset \R^m.$ Note that the representation of the density in \eqref{RCP:ass:density} is often referred to as the natural parametrization of an exponential family.\par 
In our work, we want to test the null hypothesis ``no change point''
\[H_0: \quad \theta_1 = \cdots = \theta_n\]
against the alternative ``there exists one change point''
\begin{align*}
H_1: \quad &\text{There exists an } k^*_n \in \{1, \cdots, n-1\} \text{ such that}\\
&\theta_1 = \cdots = \theta_{k^*_n} \neq \theta_{k^*_n+1} = \cdots = \theta_n.
\end{align*}
This is a so-called ``at most one change point'' (AMOC) model. Such a model is frequently studied in the literature (cf. e.g. \cite{H93, GH94, GH96a, GH96b, CH97, AHHR09}) provided that the true location of a change point $k^*_n$ is unknown but deterministic. Following Csörg\H{o} and Horváth \cite{CH97}, a natural approach to build an appropriate test statistic is based on the likelihood ratio, i.e., if the change point occurs at $k = k^*_n$ known, then we should reject $H_0$ for small values of $\Lambda_k,$ where
\begin{equation}\label{RCP:def:LR}
\Lambda_k := \frac{\sup_{\theta_0 \in \Theta} \prod_{1\leq i \leq n} f(X_i; \theta_0)}{\sup_{\theta_0^{(1)}, \theta_0^{(2)}\in \Theta} \prod_{1\leq i \leq k} f(X_i; \theta^{(1)}_0) \prod_{k < i \leq n} f(X_i; \theta^{(2)}_0)} \in (0,1].
\end{equation}

\begin{Rmk}
    In the definition of the likelihood ratio in \eqref{RCP:def:LR}, we follow the notation in \cite{CH97}. Note however, that in several other literature, the likelihood ratio has been introduced by $(\Lambda_k)^{-1}$. Then, of course, $(\Lambda_k)^{-1} \in [1, \infty)$ and the null hypothesis $H_0$ should be rejected for large values of $(\Lambda_k)^{-1}.$
\end{Rmk}

In order to guarantee the existence of the maximum likelihood estimators and later, when studying their asymptotics, we need some additional regularity assumptions.

\begin{Ass}\label{RCP:ass:reg}
    There exists an open set $\Theta_0 \subset \Theta \subset \R^d$ such that for all $\theta = (\theta_1, \cdots, \theta_d)^T\in \Theta_0,$ we have
    \begin{enumerate}
        \item [i)] $A(\theta)$ has continuous derivatives up to the third order and $A''(\theta) :=  \Big\{\frac{\partial^2}{\partial \theta_i \partial \theta_j} A(\theta),\\1\leq i, j \leq d\Big\}$ is a positive definite matrix.
        \item [ii)] $\inv A'(\theta),$ the unique inverse of $\vartheta \mapsto A'(\vartheta):= \left(\frac{\partial}{\partial \vartheta_1} A(\vartheta), \cdots, \frac{\partial}{\partial \vartheta_d} A(\vartheta)\right)^T$ at $\theta,$ exists.
    \end{enumerate}
\end{Ass}

Under Assumption \ref{RCP:ass:reg} ii) we can find unique maximum likelihood estimators (MLEs) for the parameters before and after the change provided that their true values are contained in $\Theta_0$. Elementary calculations reveal that for each $k = 1,\cdots, n-1$, the MLEs for the parameters before and after a change point $k$ are given by $\inv A'(B_n(k))$ and $\inv A'(B^*_n(k)),$ respectively, where
\[B_n(k) := \frac{1}{k}\sum_{1\leq i \leq k} T(X_i), \quad \text{ and } \quad B^*_n(k) := \frac{1}{n-k}\sum_{k < i \leq n} T(X_i).\]
Plugging in these estimators into \eqref{RCP:def:LR}, we can rewrite the log-likelihood ratio as
\begin{equation}\label{RCP:def:testStat}
S_n(k) := -\log\Lambda_k = kH(B_n(k)) + (n-k) H(B^*_n(k)) - nH(B_n(n)),
\end{equation}
where 
\begin{equation}\label{RCP:def:H}
H(x):= (\inv A'(x))^Tx - A(\inv A'(x)).
\end{equation}

\begin{Rmk}\label{RCP:rmk:derivativeH}
    Note that under Assumption \ref{RCP:ass:reg} for all $x\in \Theta_0$, the derivatives of $H$ up to the third order exists, are continuous in $x$, and
    \[H'(x) = \inv A'(x), \quad H''(x) = \left(A''(H'(x))\right)^{-1}.\]
\end{Rmk}
Since the true location of a change point $k^*_n$ is unknown, it is natural to use the maximally selected log-likelihood ratio and reject $H_0$, if
\[\mathcal{S}_n := \max_{1\leq k \leq n}\{2S_n(k)\}\]
is large.\par

In our work, under $H_1$ ``there exists one change point'', we will assume that the true location of the change point $k^*_n: \Omega \rightarrow \{1,\cdots, n-1\}$ is a random variable. This new framework is of particular interest if the change point is caused due to the occurrence of a stopping time, often depending on the data $X_1,\cdots, X_n$ itself as we already discussed in Section \ref{RCP:sec:introduction}. In the following, we will study the convergence rate and asymptotic distribution of the test statistic $\mathcal{S}_n$ under the null and under the alternative hypothesis. Moreover, under the alternative, we introduce an estimator $\hat{k}_n$ of $k^*_n$ and for the estimator of the fractional change point $\hat{\lambda}_n := \hat{k}_n/n$ of $\lambda^*_n := k^*_n/n$ we establish the consistency, the rate of convergence, and the limit distribution.

\section{Asymptotics under the null}

Under the null hypothesis $H_0$, since there is no change point in the data, we may consult the result in \cite[Theorem 1.1]{GH96a} on a limit theorem for the distribution of $\mathcal{S}_n := \max_{1\leq k \leq n} \{2S_n(k)\}$ under $H_0.$ Note that this result is a corollary of the more general result in \cite{GH94} as we restrict our considerations to densities of exponential form (cf. the assumption in \eqref{RCP:ass:density}). Let $a(x) := (2\log(x))^{1/2},$ 
\[b_d(x) := 2\log(x) + \frac{d}{2}\log\log(x) - \log(\Gamma(d/2)),\]
and $\Gamma(t) := \int_0^{\infty} y^{t-1}e^{-y}dy$ be the Gamma function.

\begin{The}[Asymptotics under the null hypothesis]\label{RCP:res:distH0} Let $H_0$ and Assumption \ref{RCP:ass:reg} be satisfied. Moreover, let $\theta_0 \in \Theta$ be the true value of the parameter in \eqref{RCP:ass:density} which is contained in $\Theta_0$. Then, for all $t \in \R,$
\[\lim_{n\rightarrow \infty} \Pro\left[a(\log n) \mathcal{S}_n^{1/2} \leq t + b_d(\log n)\right] = \exp(-2e^{-t}).\]
\end{The}

We omit the proof. The statement can be found in \cite[Theorem 1.1]{GH96a}, whereas the proof in a more general setting is stated in \cite{GH94}. \par 
The above theorem states, under the null hypothesis $H_0$ when no change point occurs, that asymptotically the test statistic $\mathcal{S}_n^{1/2}$ follows a Gumbel distribution. This is not surprising as the Gumbel distribution describes the maximum (or minimum) of normally distributed data and therefore suggests that under the null hypothesis $(S_n^{1/2}(k), 1 \leq k \leq n)$ converges to a sequences of normally distributed random variables.\par 
With help of Theorem \ref{RCP:res:distH0}, we are able to derive rejection regions of the test statistic $\mathcal{S}_n := \max_{1\leq k \leq n} \{2S_n(k)\}$ under the null hypothesis. For example, let $m = 1,$ $d = 2,$ $n = 10,000,$ and consider different significance levels $\alpha = 0.1, \, 0.05, \, 0.01,$ where $1-\alpha = \exp(-2e^{-t})$ for appropriate $t\in \R.$ Then, we should reject the null hypothesis if $\mathcal{S}_n^{1/2}$ is larger than the corresponding critical value $\kappa_{\alpha}$ calculated from equation $\Pro[\mathcal{S}_n^{1/2} > \kappa_{\alpha}] = 1-\exp(-2e^{-t}) = \alpha$ and the distribution in Theorem \ref{RCP:res:distH0}. In the table below, we have presented the critical values $\kappa_{\alpha}$ for different values of $\alpha.$

\begin{figure}[H]
    \centering
\begin{tabular}{c|c}
    $\alpha$ & $\kappa_{\alpha}$  \\
    \hline
    0.1 & 3.8827 \\
    0.05 & 4.2242\\
    0.01 & 4.9977\\
\end{tabular}
\caption{Depiction of the critical values $\kappa_{\alpha}$ for different values of $\alpha \in \{0.1,\ 0.05,\ 0.01\}$}
\label{RCP:tab:critValues}
\end{figure}

The rate of convergence to the Gumbel distribution in Theorem \ref{RCP:res:distH0} is usually believed to be very slow. Consequently, a very large sample is necessary to test $H_0$ versus $H_1$ with help of Theorem \ref{RCP:res:distH0}. In a simulation study, the authors in \cite{CH97} showed that for a moderate sample size the critical values derived from Theorem \ref{RCP:res:distH0} tend to be much larger than the true ones and therefore, the distribution in Theorem \ref{RCP:res:distH0} yields conservative rejection regions. For this reason, the authors in \cite{CH97} state a second limit theorem for the distribution of the test statistic $\mathcal{S}_n$ under the null hypothesis. In this second limit result, it is shown that the distribution of the test statistic can be approximated by that of the supremum of the continuous-time process $(B^{(d)}(t)/(t(1-t)))_{t\in (0,1)}$ taken over a slightly shorter time interval, where $B^{(d)}(t) := \sum_{1\leq i \leq d} B^2_i(t)$ and $B_1,\cdots, B_d$ are independent Brownian bridges (c.f. Theorem 1.3.2 in \cite{CH97}). Moreover, they showed that the critical values obtained from the distribution in Theorem 1.3.2 in \cite{CH97} are often preferable to the ones obtained from the distribution in Theorem \ref{RCP:res:distH0}.

\section{Asymptotics under the alternative}

Assume that $H_1$ ``there exists one change point'' holds true and let us denote by $\theta^{(1)}_0, \theta^{(2)}_0 \in \Theta \subset \R^d$ the true values of the parameters before and after the change point of $X_1,\cdots, X_n$. Let $(X_{1,i}, i \geq 1)$ and $(X_{2,i}, i \geq 1)$ be two independent sequences of iid random variables, where $X_{1,1} \sim f(x; \theta^{(1)}_0)$ and  $X_{2,1} \sim f(x; \theta^{(2)}_0).$ Then, we have 
\[X_i = \begin{cases}
X_{1,i} \quad \text{ for } i = 1,\cdots, k_n^*\\
X_{2,i} \quad \text{ for } i = k^*_n+1, \cdots, n
\end{cases}.\]
Since the densities of the $X_i$'s are elements of an exponential family, we have
\begin{equation}\label{RCP:eq:meanvarianceT}
\begin{split}
\frac{\partial A(\theta_0^{(1)})}{\partial \theta_{j}} = \E[T_j(X_{1,1})], \quad \frac{\partial^2 A(\theta^{(1)}_0)}{\partial \theta_i \partial \theta_j} &= \Cov(T_i(X_{1,1}),T_j(X_{1,1})),\\
\frac{\partial A(\theta_0^{(2)})}{\partial \theta_j} = \E[T_j(X_{2,1})], \quad \frac{\partial^2 A(\theta^{(2)}_0)}{\partial \theta_i \partial \theta_j} &= \Cov(T_i(X_{2,1}), T_j(X_{2,1})),
\end{split}
\end{equation}
where $\theta = (\theta_{1}, \cdots, \theta_d)^T \in \Theta \subset \R^d.$ In the following, we introduce by 
\begin{align}\label{RCP:def:tau}
\tau_1 := A'(\theta^{(1)}_0),\quad \tau_2 := A'(\theta^{(2)}_0),\quad \Sigma_1 := A''(\theta^{(1)}_0), \quad \text{and} \quad \Sigma_2 := A''(\theta^{(2)}_0).
\end{align}
In order to study the asymptotics of the statistic $\mathcal{S}_n$ for a possibly random occurrence of a change point $k^*_n,$ we will study its asymptotics for all possible true values of $k^*_n.$ Therefore, under $H_1,$ we can rewrite the test statistic $S_n$ in \eqref{RCP:def:testStat} as a discrete-time process of two time parameters $k, k^* \in \{1,\cdots, n-1\},$ i.e.,
\begin{equation}\label{RCP:def:testStatH1}
    S_n(k, k^*) = kH(B_n(k,k^*)) + (n-k)H(B^*_n(k,k^*)) - nH(B_n(n,k^*)),
\end{equation}
where $H$ is given as in \eqref{RCP:def:H},
\[B_n(k,k^*) = \begin{cases}
    \frac{1}{k}\sum_{i = 1}^k T(X_{1,i}) &\quad \text{ if } k \leq k^*\\
    \frac{1}{k}\left(\sum_{1=1}^{k^*}T(X_{1,i}) + \sum_{i = k^* +1}^kT(X_{2,i})\right) &\quad \text{ if } k > k^*
\end{cases},\]
and
\[B^*_n(k,k^*) = \begin{cases}
    \frac{1}{n-k}\left(\sum_{i = k+1}^{k^*} T(X_{1,i}) + \sum_{i = k^* +1}^n T(X_{2,i})\right) &\quad \text{ if } k\leq k^*\\
    \frac{1}{n-k}\sum_{i = k+1}^n T(X_{2,i}) &\quad \text{ if } k > k^*
\end{cases}.\]

In the following, we will study the limit distribution of $S_n$. It turns out that its distribution depends on the limit distribution of $k^*_n/n$ and on the size of  the change 
\begin{align}\label{RCP:def:DeltaSquared}
    \Delta^2 := \|\theta^{(1)}_0-\theta^{(2)}_0\|^2.
\end{align}

\begin{Ass}\label{RCP:ass:cases}
    Let the true location of the change point $k^*_n$ be a random variable taking values in $\{1,\cdots, n-1\}.$ Let $\gamma \in (0,1/2)$ be a constant and $\lambda^*$ be a random variable taking values in $[\gamma,1-\gamma]$ with probability one such that $k^*_n/n \Rightarrow \lambda^*.$ Moreover, we assume that $\theta^{(1)}_0 := \theta^{(1)}_0(n) \rightarrow \theta_A$ and $\theta^{(2)}_0 := \theta^{(2)}_0(n) \rightarrow \theta_A$ for some $\theta_A$ in the interior of $\Theta \subset \R^d$ with
        \begin{equation}\label{RCP:ass:sizeCP}
        \lim_{n\rightarrow \infty} n \Delta^2 = \infty,
        \end{equation}
    where the size of the change $\Delta^2$ is given in \eqref{RCP:def:DeltaSquared}.
\end{Ass}

Assuming that $k^*_n/n \Rightarrow \lambda^*$ for some random variable $\lambda^*$ taking values almost surely in $[\gamma, 1-\gamma],$ for $\gamma \in (0,1/2),$ ensures that the change point occurs in ``the middle of the data''. Moreover, we concentrate in the following on the critical case in which the size of the change point $\Delta^2$ converges to zero as $n\rightarrow \infty.$  We will show that if $\Delta^2$ ensures the condition in \eqref{RCP:ass:sizeCP}, we are still able to detect the change point in the data. 

\begin{Rmk}
    Csörg\H{o} and Horváth \cite{CH97} studied this problem provided that $k^*_n$ is deterministic and $k^*_n/n \rightarrow \lambda \in (0,1).$ Moreover, they studied slight modifications of Assumption \ref{RCP:ass:cases}, e.g.,
        \begin{itemize}
            \item[i)] the occurrence of an early change point, i.e., $k^*_n/n \rightarrow 0,$ and
            \item[ii)] the size of the change point is large compared to the sample size in the sense that $\Delta^2$ is independent of $n.$
        \end{itemize}
        Combining our subsequent analysis with the arguments in \cite{CH97}, we expect to derive similar results, also in these settings.
\end{Rmk}

\begin{Rmk}
    If Assumption \ref{RCP:ass:cases} is satisfied, the true values of the parameters $\theta^{(1)}_0,$ $\theta^{(2)}_0,$ and hence also the true values of the parameters $\tau_1,$ $\tau_2,$ $\Sigma_1,$ $\Sigma_2,$ and $\Delta^2$ introduced in \eqref{RCP:def:tau} and \eqref{RCP:def:DeltaSquared} depend on $n.$ However, for reasons of notation, this dependence will often be omitted.
\end{Rmk}

Since the sequences $(X_{1,i}, i\geq 1)$ and $(X_{2,i}, i\geq 1)$ contain independent and identically distributed random variables, we can state our first limit theorem. It is a direct consequence of Donsker's theorem in higher dimensions.

\begin{Lem}\label{RCP:res:clt1}
Let Assumption \ref{RCP:ass:reg} and \ref{RCP:ass:cases} be satisfied and $\theta_A \in \Theta_0.$ Moreover, we define by $W^{(n)}_l := (W^{(n)}_l(t))_{t\in [0,1]}$, where $W^{(n)}_l(t) := \sum_{k=1}^{n}W^{(n)}_{l,k}\1_{\{nt \in [k, k+1)\}},$ $W^{(n)}_{l,k} := n^{-1/2} \sum_{i=1}^k (T(X_{l,i})-\tau_l),$ and $l = 1,2.$ Then, for each $l\in \{1,2\},$ we have 
\[W^{(n)}_l \Rightarrow \Sigma_A^{1/2}W\]
in the Skorokhod topology on $D([0,1],\R^d)$, where $\Sigma_A := A''(\theta_A).$ Here, $W$ denotes a $d$-dimensional standard Brownian motion.
\end{Lem}

\begin{proof}
    This is a direct application of Donsker's theorem in higher dimension to the iid sequences $(T(X_{1,i})-\tau_1,\, i \geq 1)$ or $(T(X_{2,i}) - \tau_2,\, i \geq 1).$
\end{proof}

Next, for all $k,k^* \in \{1,\cdots, n-1\},$ let us introduce
\begin{equation}\label{RCP:def:asyMean}
\mu_n(k,k^*) := \begin{cases}
    kH(\tau_1) + (n-k)H\left(\frac{k^*-k}{n-k} \tau_1 + \frac{n-k^*}{n-k}\tau_2\right)\\
    \hspace{3cm} - nH\left(\frac{k^*}{n}\tau_1 + \frac{n-k^*}{n}\tau_2\right) &\quad \text{ if } k \leq k^*\\
    kH\left(\frac{k^*}{k}\tau_1 + \frac{k-k^*}{k}\tau_2\right) + (n-k)H(\tau_2)\\
    \hspace{3cm} - nH\left(\frac{k^*}{n}\tau_1 + \frac{n-k^*}{n}\tau_2\right) &\quad \text{ if } k> k^*
\end{cases}.
\end{equation}
It turns out that $\mu_n(k,k^*)$ is the expected value of the statistic $S_n(k,k^*)$. Then, for all $k, k^* \in \{1,\cdots, n-1\},$ applying Taylor's formula of the first order, we can write
\begin{equation}\label{RCP:eq:testStatH1-centered}
S_n(k,k^*) - \mu_n(k,k^*) = Z_n(k, k^*) + R_n(k,k^*),
\end{equation}
where for $h_n(x) := H'(x\tau_1(n) + (1-x)\tau_2(n))^T,$ $k \leq k^*,$
\begin{equation}\label{RCP:eq:defZn1}
    \begin{split}
        Z_n(k,k^*) &:= h_n(1) \sum_{i=1}^k (T(X_{1,i}) - \tau_1)\\
            &\quad + h_n\left(\frac{k^*-k}{n-k}\right) \left(\sum_{i = k+1}^{k^*} (T(X_{1,i}) - \tau_1) + \sum_{i=k^*+1}^n (T(X_{2,i}) - \tau_2)\right)\\
            &\quad - h_n\left(\frac{k^*}{n}\right) \left(\sum_{i=1}^{k^*} (T(X_{1,i}) -\tau_1) + \sum_{i = k^*+1}^n (T(X_{2,i}) - \tau_2)\right),
    \end{split}
\end{equation}
for $k > k^*,$
\begin{equation}\label{RCP:eq:defZn2}
    \begin{split}
        Z_n(k,k^*)&:=h_n\left(\frac{k^*}{k}\right) \left(\sum_{i=1}^{k^*}(T(X_{1,i}) - \tau_1) + \sum_{i=k^*+1}^k (T(X_{2,i}) -\tau_2)\right)\\
        &\quad + h_n(0) \sum_{i=k+1}^n (T(X_{2,i}) -\tau_2)\\
        &\quad -h_n\left(\frac{k^*}{n}\right) \left(\sum_{i=1}^{k^*} (T(X_{1,i}) - \tau_1) + \sum_{i = k^* +1}^n(T(X_{2,i}) - \tau_2)\right),
    \end{split}
\end{equation}
and $R_n(k,k^*)$ is the corresponding remainder of Lagrange form such that the equation in \eqref{RCP:eq:testStatH1-centered} holds true. Later, we will see that under the appropriate rescaling such that the piecewise constant interpolation of $(Z_n(k,k^*); k,k^* \in \{1,\cdots, n-1\})$ converges weakly in the Skorokhod topology to a non-trivial stochastic limit process, the piecewise constant interpolation of $(R_n(k,k^*); k,k^* \in \{1,\cdots, n-1\})$ will vanish in probability as $n\rightarrow \infty$. With a little abuse of notation, we define by $Z_n := (Z_n(t,\lambda))_{t,\lambda \in [0,1]}$ the piecewise constant interpolation of $(Z_n(k,k^*); k,k^* \in \{1,\cdots, n-1\}),$ where
\[Z_n(t,\lambda) := \sum_{k=1}^{n-1}\sum_{k^*=1}^{n-1} Z_n(k,k^*) \1_{\{nt \in [k, k+1)\}}\1_{\{n\lambda \in [k^*, k^*+1)\}}.\]
Similarly, we define $\mu_n := (\mu_n(t,\lambda))_{t,\lambda \in [0,1]}$ and $R_n := (R_n(t,\lambda))_{t,\lambda \in [0,1]}.$

In the following, we assume for simplicity that $f(x,\theta)$ is the density of an $m$-dimensional normal distribution given in its natural parametrization, i.e., for mean $\mu \in \R^m$ and covariance matrix $\Sigma \in \R^{m\times m}$ symmetric and positive definite, let $\theta := \theta(\mu, \Sigma) = (\Sigma^{-1}\mu, -\frac{1}{2}\Sigma^{-1}) =: (\theta_1, \theta_2) \in \R^m \times \R^{m\times m}.$ Then, the density in \eqref{RCP:ass:density} is given by
\[f(x;\theta) = \exp\left(\theta^T_1 x + x^T \theta_2 x + \frac{1}{4}\theta^T_1 \theta^{-1}_2 \theta_1 - \frac{1}{2}\log(\det(-\pi\theta^{-1}_2))\right),\]
and therefore $T(x) = (x, xx^T),$ $A(\theta) = -\frac{1}{4}\theta^T_1 \theta^{-1}_2 \theta_1 + \frac{1}{2}\log(\det(-\pi\theta_2^{-1}))$ for $x\in \R^m,$ $\theta = (\theta_1, \theta_2) \in \R^d = R^{m+m^2},$ and $H(y) = -\frac{1}{2}\det(2\pi(y_2-y_1y_1^T)),$ for $y = (y_1, y_2) \in \R^m \times \R^{m\times m}.$

\begin{Rmk}\label{RCP:rmk:derivativesNormal}
    Since $f(x, \theta)$ is the density of an $m$-dimensional normal distribution, we have $d = m+m^2.$ Moreover, Assumption \ref{RCP:ass:reg} is satisfied.
\end{Rmk}

Assuming that $f(x,\theta)$ is the density of an $m$-dimensional normal distribution will remarkably simplify the proof of the following limit theorem as the identification of the finite-dimensional distributions can be derived by the fourth moment theorem (cf. Nualart and Peccati \cite[Theorem 1]{NP05}).

\begin{The}[A limit theorem for $Z_n$ under the alternative]\label{RCP:res:limitZn1}
Let Assumption \ref{RCP:ass:cases} be satisfied and assume that $f(x, \theta)$ is the density of an $m$-dimensional normal distribution given in its natural parametrization. Let us denote by $\delta^2 := \delta^2(n) := \|\tau_1(n) - \tau_2(n)\|^2\rightarrow 0$ as $n\rightarrow \infty$ and by $\tau_A := \lim_{n\rightarrow \infty} \tau_1(n).$ Then, we have
\[(n\delta^2)^{-1/2}Z_n \Rightarrow Z^*\]
in the Skorokhod topology on $D([0,1]^2, \R)$, where $Z^*$ is a Gaussian process with mean zero and covariance function
\[c((t,\lambda),(t',\lambda')) = \sigma^2_A \begin{cases}
(1-\lambda)(1-\lambda') \min\left\{\frac{t}{1-t}, \frac{t'}{1-t'}\right\}, & \text{ if } t \leq \lambda, t' \leq \lambda'\\
(1-\lambda)\lambda'\min\left\{\frac{t(1-t')}{(1-t)t'}, 1\right\}, & \text{ if } t \leq \lambda, t' > \lambda'\\
\lambda(1-\lambda') \min\left\{\frac{(1-t)t'}{t(1-t')}, 1\right\}, & \text{ if } t > \lambda, t' \leq \lambda'\\
\lambda \lambda' \min\left\{\frac{1-t}{t}, \frac{1-t'}{t'}\right\}, & \text{ if } t > \lambda, t' > \lambda'
\end{cases},\]
for $((t, \lambda), (t',\lambda')) \in [0,1]^2\times [0,1]^2 \setminus \{((1,1), (1,1))\}$ and $c((1,1),(1,1)) = 0,$ and $\sigma_A^2$ given by
\[\sigma^2_A := \lim_{n\rightarrow \infty} \sigma^2_A(n) := \lim_{n\rightarrow \infty} \frac{(\tau_1(n) - \tau_2(n))^TH''(\tau_A) (\tau_1(n) - \tau_2(n))}{\|\tau_1(n) -\tau_2(n)\|^2}\]
provided that the limit on the right hand-side exists. Otherwise, we simply scale the left hand-side appropriately with $\sigma_A^2(n).$
\end{The}

\begin{Cor}
    Let the assumptions of Theorem \ref{RCP:res:limitZn1} be satisfied and assume that $\sigma^2_A$ introduced in Theorem \ref{RCP:res:limitZn1} exists. Then, we have 
        \[\left(\frac{Z_n(t,t)}{\sqrt{n\delta^2}}\right)_{t \in [0,1]} \Rightarrow \sigma_A B\]
    in the Skorokhod topology on the space $D([0,1], \R)$, where $B$ is a one-dimensional Brownian bridge, i.e., $B$ is a Gaussian process with mean zero and covariance function $c(t,t') = \min \{t, t'\} - tt'$.
\end{Cor}

\begin{figure}[H]
    \centering
    \includegraphics[scale = 0.3]{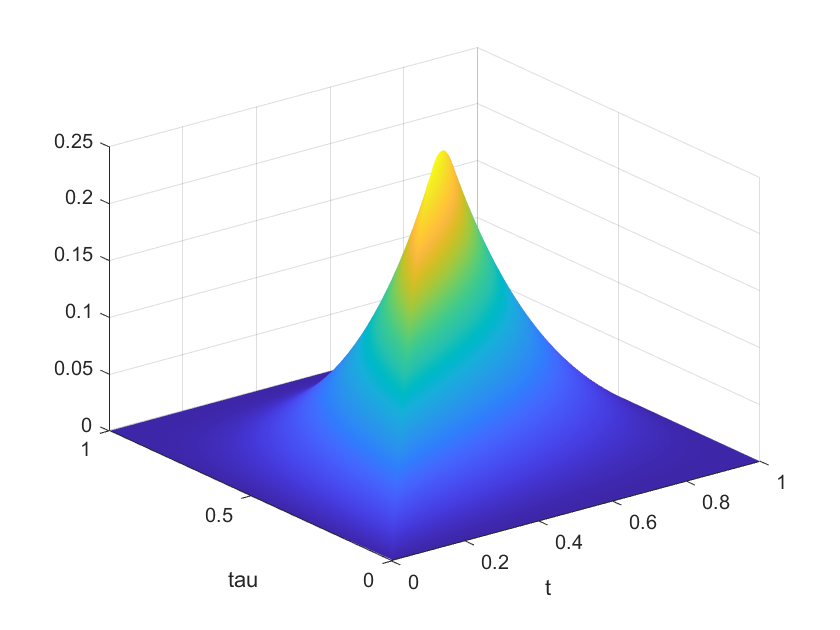}
    \caption{Depiction of the variance structure of the limit process derived in Theorem \ref{RCP:res:limitZn1}.}
    \label{RCP:fig:covZn1}
\end{figure}

\begin{proof}[Proof of Theorem \ref{RCP:res:limitZn1}]
    For $l = 1,2$, recall that $W_l^{(n)}(t) := \sum_{k=1}^nW_{l,k}^{(n)}\1_{\{nt \in [k,k+1)\}},$  where $W^{(n)}_{l,k} := n^{-1/2}\sum_{i =1}^k (T(X_{l,i}) - \tau_l)$ and let $\tilde{Z}_n(t,\lambda):= (n\delta^2)^{-1/2} Z_n(t,\lambda)$ for $t,\lambda \in [0,1].$ In the following, we assume the existence of the limit $\sigma^2_A$ given in Theorem \ref{RCP:res:limitZn1}. Otherwise, we simply study the process $\bar{Z}_n(t,\lambda) = (\sigma^2_A(n))^{-1} \tilde{Z}_n.$ \par 
    We will first establish tightness of the sequence $(\tilde{Z}_n)_{n\in \N}$ and then prove that their finite-dimensional distributions converge to those of $Z^*$.\newline 
    
    \textit{Tightness:} For $t,\lambda \in [0,1]$, let us introduce the short-hand notations \[\alpha_n(t,\lambda):= \frac{\lfloor n\lambda \rfloor - \lfloor nt \rfloor}{n-\lfloor nt \rfloor}\1_{\{t < \lambda\}}, \quad \beta_n(\lambda):= \frac{\lfloor n\lambda \rfloor}{n}, \quad \text{ and } \quad \gamma_n(t,\lambda):= \frac{\lfloor n\lambda \rfloor}{\lfloor nt \rfloor}\1_{\{t > \lambda\}}.\]
    As $n\rightarrow \infty,$ we have for each $t, \lambda \in [0,1],$ that $\alpha_n(t,\lambda) \rightarrow (\lambda - t)/(1-t)\1_{\{t< \lambda\}},$ $\beta_n(\lambda) \rightarrow \lambda,$ and $\gamma_n(t,\lambda) \rightarrow \lambda/t\1_{\{t > \lambda\}}.$ Now, for $t \leq \lambda$ and $\alpha_n:= \alpha_n(t,\lambda),$ $\beta_n := \beta_n(\lambda),$ we can write
    \begin{equation*}
        \begin{split}
          Z_n(t,\lambda) &= \left(h_n(1)-h_n(\beta_n)\right)n^{1/2} W^{(n)}_{1}(\lambda) + \left(h_n(\alpha_n) - h_n(1)\right)n^{1/2} \left(W^{(n)}_1(\lambda) - W^{(n)}_1(t)\right)\\
          &\qquad+ (h_n(\alpha_n)-h_n(\beta_n)) n^{1/2} \left(W^{(n)}_2(1) - W^{(n)}_2(\lambda)\right)
        \end{split}
    \end{equation*}
    and for $t > \lambda$ and $\gamma_n := \gamma_n(t,\lambda),$ 
    \begin{equation*}
        \begin{split}
          Z_n(t,\lambda) &= \left(h_n(\gamma_n)-h_n(\beta_n)\right) n^{1/2}W_1^{(n)}(\lambda) + \left(h_n(0) - h_n(\beta_n)\right) n^{1/2}\left(W^{(n)}_{2}(1) - W^{(n)}_2(\lambda)\right)\\
          &\qquad+ (h_n(\gamma_n)-h_n(0)) n^{1/2} \left(W^{(n)}_{2}(t) - W^{(n)}_2(\lambda)\right).
        \end{split}
    \end{equation*}
    Since $f(x,\theta)$ is the density of an $m$-dimensional normal distribution, we conclude $H$ has continuous derivatives up to the second order (cf. Remark \ref{RCP:rmk:derivativesNormal}) and by Assumption \ref{RCP:ass:cases}, $(\tau_1, \tau_2)\rightarrow (\tau_A, \tau_A)$ as $n\rightarrow \infty,$ where $\tau_A := A'(\theta_A).$ Hence, for all sequences $(x_n)_{n\in \N}, (y_n)_{n\in \N} \subset [0,1]$ with $x_n \rightarrow x \in [0,1],$ $y_n \rightarrow y \in [0,1]$ as $n\rightarrow \infty,$ and $x_n\leq y_n$ for all $n\in \N,$ we conclude by an application of the definition of the derivative in higher dimensions and Remark \ref{RCP:rmk:derivativeH} that
    \begin{equation}\label{RCP:eq:convGn}
        \begin{split}
            \frac{h_n(y_n) - h_n(x_n)}{\delta}
            &= \frac{\left(H'(y_n \tau_1 + (1-y_n)\tau_2) - H'(x_n \tau_1 + (1-x_n)\tau_2)\right)^T}{\|\tau_1 - \tau_2\|}\\
            &= (y_n - x_n)\frac{(\tau_1 - \tau_2)^T}{\|\tau_1 - \tau_2\|}H''(\tau_2) + o(1)\\
            &\rightarrow (y-x)\tilde{\Sigma}^{1/2}
        \end{split}
    \end{equation}
    as $n\rightarrow \infty,$ where $\tilde{\Sigma}^{1/2} := \lim_{n\rightarrow \infty} (\tau_1-\tau_2)^T H''(\tau_A)/\delta.$
    In the following, we denote by $G_n(x,y)=(h_n(x)-h_n(y))/\delta.$ Then, the process $\tilde{Z}_n$ can be written as follows: for $t\leq \lambda,$
    \begin{align*}
        \tilde{Z}_n(t,\lambda) &= G_n(1, \beta_n) W^{(n)}_1(\lambda) - G_n(1,\alpha_n) W^{(n)}_{1}(\lambda)  + G_n(1,\alpha_n) W^{(n)}_1(t)\\
        &\qquad - G_n(\beta_n,\alpha_n) W^{(n)}_2(1) + G_n(\beta_n, \alpha_n)W^{(n)}_2(\lambda),
    \end{align*}
    and for $t > \lambda,$
    \begin{align*}
        \tilde{Z}_n(t,\lambda)&= G_n(\gamma_n, \beta_n) W^{(n)}_1(\lambda) - G_n(\beta_n, 0) W^{(n)}_2(1) + G_n(\beta_n, 0) W^{(n)}_2(\lambda)\\
        &\quad+ G_n(\gamma_n, 0) W^{(n)}_2(t) - G_n(\gamma_n, 0) W^{(n)}_2(\lambda).
    \end{align*}
    Hence, the process $\tilde{Z}_n$ can be represented as a sum of five scalar products which are derived by multiplying the $d$-dimensional partial sums $W^{(n)}_l(t),$ $W^{(n)}_l(\lambda),$ or $W^{(n)}_l(1),$ $l = 1,2,$ with the non-random $d$-dimensional vector $(G_n \circ (x_n,y_n))(t,\lambda),$ where $x_n(t,\lambda),$ $y_n(t,\lambda) \in \{0, \alpha_n(t,\lambda), \beta_n(\lambda),$ $\gamma_n(t,\lambda), 1\}.$ Therefore, we can identify each summand of $\tilde{Z}_n$ as a discrete-time process in two time parameters which converges weakly in the Skorokhod topology to a one-dimensional Gaussian process thanks to Lemma \ref{RCP:res:clt1} and \eqref{RCP:eq:convGn}. In particular, the limit process of each summand of $\tilde{Z}_n$ is a continuous process in both time parameters. By Theorem 12.6.1 in Whitt \cite{W02}, we conclude their joint convergence implying that the sequence $(\tilde{Z}_n)_{n\in \N}$ is tight.
    
    \textit{Convergence of the finite-dimensional distributions:} It is left prove the convergence of the finite-dimensional distributions of $\tilde{Z}_n$, i.e., for all $k \geq 1$ and $(t_1, \lambda_1), \cdots, (t_k, \lambda_k) \in [0,1]^2,$ we want to show that
    \[(\tilde{Z}_n(t_1,\lambda_1),\cdots, \tilde{Z}_n(t_k, \lambda_k)) \Rightarrow (Z^*(t_1, \lambda_1),\cdots, Z^*(t_k, \lambda_k))\]
    as $n\rightarrow \infty,$ where $Z^*$ is a Gaussian process with mean zero and  covariance function $c: [0,1]^2 \times [0,1]^2 \rightarrow \R$ defined in Theorem \ref{RCP:res:limitZn1}. For sake of notation, we only analyze the case $k = 2,$ i.e., for $(t,\lambda), (t',\lambda') \in [0,1]^2,$ we will prove the joint convergence of 
    \begin{equation}\label{RCP:eq:jointConv}
    (\tilde{Z}_n(t,\lambda), \tilde{Z}_n(t',\lambda')) \Rightarrow (Z^*(t,\lambda), Z^*(t',\lambda'))
    \end{equation}
    as $n\rightarrow \infty$ and note that for $k>2$, we can argue completely analogously. Applying the Cramér-Wold device, \eqref{RCP:eq:jointConv} is equivalent to 
    \begin{equation}\label{RCP:eq:CWD}
        \hat{Z}_n((t,\lambda),(t',\lambda')) := x\tilde{Z}_n(t,\lambda) + y\tilde{Z}_n(t',\lambda') \Rightarrow xZ^*(t,\lambda) + yZ^*(t',\lambda'),
    \end{equation}
    where $x,y \in \R$ are arbitrary. Since the $X_i$'s are assumed to be normally distributed, the first $m$ components of $W^{(n)}_l$ belong to the Wiener chaos of order one and the last $m^2$ components of $W^{(n)}_l$ belong to the Wiener chaos of order two, for $l = 1,2$ and all $n\in \N$. Hence, $\hat{Z}_n$ belongs to the Wiener chaos of order two for all $n\in \N.$ In order to prove \eqref{RCP:eq:CWD}, we will apply the fourth moment theorem by Nualart and Peccati \cite{NP05}. According to the fourth moment theorem, for all $(t,\lambda), (t',\lambda') \in [0,1]^2,$ the convergence in \eqref{RCP:eq:CWD} is satisfied if the following two conditions hold true: as $n\rightarrow \infty,$ we have 
    \begin{itemize}
        \item[i)] $\Var[\hat{Z}_n((t,\lambda),(t',\lambda'))] \rightarrow \hat{c}((t,\lambda),(t',\lambda'))$ and
        \item[ii)] $\E[\hat{Z}^4_n((t,\lambda),(t',\lambda'))] \rightarrow 3\left(\hat{c}((t,\lambda),(t',\lambda'))\right)^2,$
    \end{itemize}
    where 
    \[\hat{c}((t,\lambda), (t',\lambda')) = x^2c((t,\lambda),(t,\lambda)) + 2xy\, c((t,\lambda),(t',\lambda')) + y^2c((t',\lambda'),(t',\lambda'))\]
    and the covariance function $c: [0,1]^2\times [0,1]^2 \rightarrow \R$ defined in Theorem \ref{RCP:res:limitZn1}.
    Since $\E[T(X_{l,1})] = \tau_l$ for $l = 1,2,$ we conclude that $\E[\hat{Z}_n((t,\lambda),(t',\lambda'))] = 0.$ In order to analyze the second and fourth moments of $\hat{Z}_n((t,\lambda),(t',\lambda'))$, we need to differentiate between the following four cases:
    \[\text{1) } t\leq \lambda, \, t' \leq \lambda', \quad \text{2) } t \leq \lambda, \, t' > \lambda',\quad \text{3) } t > \lambda, \,t' \leq \lambda',\quad \text{and} \quad \text{4) } t> \lambda, \, t' > \lambda'.\] 
    First, let 1) $((t,\lambda), (t',\lambda')) \in [0,1]^2\times [0,1]^2 \setminus \{((1,1),  (1,1))\}$ with $t \leq \lambda,$ $t' \leq \lambda'$ and denote by $\alpha_n := \alpha_n(t, \lambda),$ $\alpha'_n := \alpha_n(t',\lambda'),$ $\beta_n := \beta_n(\lambda),$ and $\beta'_n:= \beta_n(\lambda_n).$ Then, for $\sigma^2_A$ defined in Theorem \ref{RCP:res:limitZn1}, applying Lemma \ref{RCP:res:clt1} and \eqref{RCP:eq:convGn}, we conclude
    \begin{align*}
        &\E[\tilde{Z}_n(t,\lambda)\tilde{Z}_n(t',\lambda')]=G_n(1,\beta_n)\E\left[W^{(n)}_1(\lambda)\left( W^{(n)}_1(\lambda')\right)^T\right](G_n(1,\beta_n'))^T\\
        &\qquad - G_n(1,\beta_n)\E\left[W^{(n)}_1(\lambda)\left(W^{(n)}_1(\lambda') - W^{(n)}_1(t')\right)^T\right](G_n(1,\alpha_n'))^T\\
        &\qquad -G_n(1,\beta_n)\E\left[W^{(n)}_1(\lambda)\left(W^{(n)}_2(1) - W^{(n)}_2(\lambda')\right)^T\right] (G_n(\beta_n', \alpha_n'))^T\\
        &\qquad - G_n(1,\alpha_n)\E\left[\left(W^{(n)}_{1}(\lambda) - W^{(n)}_1(t)\right) \left(W^{(n)}_1(\lambda')\right)^T\right](G_n(1,\beta_n'))^T\\
        &\qquad + G_n(1,\alpha_n) \E\left[\left(W^{(n)}_1(\lambda) - W^{(n)}_1(t)\right)\left(W^{(n)}_1(\lambda') - W^{(n)}_1(t')\right)^T\right] (G_n(1,\alpha'_n))^T\\
        &\qquad + G_n(1,\alpha_n)\E\left[\left(W^{(n)}_1(\lambda) -W^{(n)}_1(t)\right)\left(W^{(n)}_2(1) -W^{(n)}_2(\lambda')\right)^T\right](G_n(\beta_n', \alpha_n'))^T\\
        &\qquad -G_n(\beta_n, \alpha_n)\E\left[\left(W^{(n)}_2(1) - W^{(n)}_2(\lambda)\right)\left(W^{(n)}_1(\lambda')\right)^T\right](G(1,\beta_n'))^T\\
        &\qquad + G_n(\beta_n, \alpha_n)\E\left[\left(W^{(n)}_2(1) - W^{(n)}_2(\lambda)\right) \left(W^{(n)}_1(\lambda') - W^{(n)}_1(t')\right)^T\right] (G(1,\alpha'_n))^T\\
        &\qquad + G_n(\beta_n, \alpha_n)\E\left[\left(W^{(n)}_2(1) - W^{(n)}_2(\lambda)\right) \left(W^{(n)}_2(1) - W^{(n)}_2(\lambda')\right)^T\right](G(\beta_n', \alpha_n'))^T\\
        &\rightarrow \sigma^2_A(1-\lambda)(1-\lambda')\Bigg(\min\{\lambda, \lambda'\} - \frac{1}{1-t'}\max\{\min\{\lambda, \lambda'\}-t', 0\} - \frac{t'}{1-t'}\max\{\lambda-\lambda',0\}\\
        &\qquad - \frac{1}{1-t}\max\{\min\{\lambda, \lambda'\} -t,0\} + \frac{1}{(1-t)(1-t')}\max\{\min\{\lambda, \lambda'\}-\max\{t,t'\}, 0\}\\
        &\qquad + \frac{t'}{(1-t)(1-t')}\max\{\lambda - \max\{t,\lambda'\}, 0\} - \frac{t}{1-t}\max\{\lambda'-\lambda, 0\}\\
        &\qquad + \frac{t}{(1-t)(1-t')}\max\{\lambda'-\max\{\lambda, t'\}, 0\} + \frac{tt'}{(1-t)(1-t')}(1-\max\{\lambda, \lambda'\})\Bigg)\\
        &=\sigma^2_A (1-\lambda)(1-\lambda') \min\left\{\frac{t}{1-t}, \frac{t'}{1-t'}\right\}.
    \end{align*}
    Moreover, by definition of $c: [0,1]^2 \times [0,1]^2 \rightarrow \R$ at $((1,1), (1,1)),$ we have $\E[\tilde{Z}^2_n(1,1)] \rightarrow 0 = c((1,1),(1,1))$ as $n\rightarrow \infty.$
    Analogously, by applying Donsker's theorem in higher dimensions for the partial sums $W^{(n)}_l,$ $l =1,2,$ (cf. Lemma \ref{RCP:res:clt1}) and \eqref{RCP:eq:convGn}, we can derive the limit covariance of $\tilde{Z}_n$ in the remaining three cases, where the cases 2) and 3) are symmetric. Finally, we conclude that for $((t,\lambda), (t',\lambda')) \in [0,1]^2 \times [0,1]^2 \setminus \{((1,1), (1,1))\},$
    \[\E\left[\tilde{Z}_n(t,\lambda)\tilde{Z}_n(t',\lambda')\right] \rightarrow \sigma^2_A \begin{cases}
(1-\lambda)(1-\lambda') \min\left\{\frac{t}{1-t}, \frac{t'}{1-t'}\right\}, & \text{ if } t \leq \lambda, t' \leq \lambda'\\
(1-\lambda)\lambda' \min\left\{\frac{t(1-t')}{(1-t)t'}, 1\right\}, & \text{ if } t \leq \lambda, t' > \lambda'\\
\lambda(1-\lambda') \min\left\{\frac{(1-t)t'}{t(1-t')}, 1\right\}, & \text{ if } t > \lambda, t' \leq \lambda'\\
\lambda \lambda' \min\left\{\frac{1-t}{t}, \frac{1-t'}{t'}\right\}, & \text{ if } t > \lambda, t' > \lambda'
\end{cases}.\]
    In particular, by definition of $\hat{Z}_n,$ we conclude for all $(t,\lambda), (t',\lambda') \in [0,1]^2$ that
    \[\Var[\hat{Z}_n((t,\lambda),(t',\lambda'))] \rightarrow \hat{c}((t,\lambda),(t',\lambda'))\]
    as desired. Next, we analyze $\E[\hat{Z}_n^4((t,\lambda),(t',\lambda'))].$ For this reason, we will calculate the mixed fourth moments $\E[\tilde{Z}_n^3(t,\lambda)\tilde{Z}_n(t',\lambda')]$ and $\E[\tilde{Z}^2_n(t,\lambda)\tilde{Z}^2_n(t',\lambda')].$ Recall that $\tilde{Z}_n$ can be represented by finitely many scalar products between the partial sums $W^{(n)}_l, l =1,2,$ and the non-random function $G_n$ evaluated at $x_n(t,\lambda), y_n(t,\lambda) \in \{0,\alpha_n(t,\lambda), \beta_n(\lambda), \gamma_n(t,\lambda), 1\}.$ Because of Lemma \ref{RCP:res:clt1}, $(W^{(n)}_l)_{n\in \N}$, $l = 1,2,$ converges weakly in the Skorokhod topology to a $d$-dimensional Brownian motion. Thus, we can apply the same arguments as for the computation of the mixed fourth moments of partial sums of iid standard normal distributed random variables.
    
    \begin{Lem}\label{RCP:res:4thMomentIIDN}
    Let $(\xi_i, i \geq 1)$ be a sequence of iid one-dimensional standard normal random variables and $W^{(n)}(t) := \sum_{k=1}^{n}W^{(n)}_k\1_{\{nt \in [k,k+1)\}},$ where $W^{(n)}_k:= n^{-1/2}\sum_{i=1}^k \xi_i.$ Moreover, let $\{\alpha_i\}_{i=1,2,3,4},$ $\{\beta_i\}_{i=1,2,3,4} \subset [0,1]$ with $\alpha_i \leq \beta_i$ for all $i = 1,2,3,4,$ and denote by $\check{\alpha} := \max\{\alpha_i : i =1,2,3,4\},$ $\check{\alpha}_{ij} := \max\{\alpha_i, \alpha_j\},$ $\hat{\beta} := \min\{\beta_i: i = 1,2,3,4\},$ and $\hat{\beta}_{ij} := \min\{\beta_i, \beta_j\}.$ Then, as $n\rightarrow \infty,$ we have
    \begin{equation*}
    \begin{split}
        \E&\left[\prod_{i=1}^4 \left(W^{(n)}(\beta_i) - W^{(n)}(\alpha_i)\right)\right]\\
        &\quad = \max\{\hat{\beta}_{12}-\check{\alpha}_{12}, 0\}\max\{\hat{\beta}_{34}-\check{\alpha}_{34}, 0\}+ \max\{\hat{\beta}_{13}-\check{\alpha}_{13}, 0\}\max\{\hat{\beta}_{24}-\check{\alpha}_{24}, 0\}\\
        &\qquad +\max\{\hat{\beta}_{14}-\check{\alpha}_{14}, 0\}\max\{\hat{\beta}_{23}-\check{\alpha}_{23}, 0\}.
    \end{split}
    \end{equation*}
    \end{Lem}
    
    \begin{proof}
        Since we assume that $(\xi_i, i \geq 1)$ is a sequence of iid standard normal random variables and by the definition of $W^{(n)},$ we conclude as $n\rightarrow \infty,$
        \begin{align*}
            \E&\left[\prod_{i=1}^4 \left(W^{(n)}(\beta_i) - W^{(n)}(\alpha_i)\right)\right]\\
            &\quad = \frac{1}{n^2}\sum_{i= \lfloor n\check{\alpha}\rfloor +1}^{\lfloor n \hat{\beta}\rfloor}\E\left[\xi_i^4\right] + \frac{1}{n^2} \sum_{i = \lfloor n\check{\alpha}_{12}\rfloor +1}^{\lfloor n\hat{\beta}_{12}\rfloor} \sum_{\substack{j = \lfloor n\check{\alpha}_{34}\rfloor +1\\ j\neq i}}^{\lfloor n\hat{\beta}_{34}\rfloor}\E\left[\xi^2_i \xi^2_j\right]\\
            &\qquad + \frac{1}{n^2} \sum_{i = \lfloor n\check{\alpha}_{13}\rfloor +1}^{\lfloor n\hat{\beta}_{13}\rfloor} \sum_{\substack{j = \lfloor n\check{\alpha}_{24}\rfloor +1\\ j\neq i}}^{\lfloor n\hat{\beta}_{24}\rfloor}\E\left[\xi^2_i \xi^2_j\right] + \frac{1}{n^2} \sum_{i = \lfloor n\check{\alpha}_{14}\rfloor +1}^{\lfloor n\hat{\beta}_{14}\rfloor} \sum_{\substack{j = \lfloor n\check{\alpha}_{23}\rfloor +1\\ j\neq i}}^{\lfloor n\hat{\beta}_{23}\rfloor}\E\left[\xi^2_i \xi^2_j\right]\\
            &\quad \rightarrow 0 + \max\{\hat{\beta}_{12}-\check{\alpha}_{12}, 0\}\max\{\hat{\beta}_{34}-\check{\alpha}_{34}, 0\}+ \max\{\hat{\beta}_{13}-\check{\alpha}_{13}, 0\}\max\{\hat{\beta}_{24}-\check{\alpha}_{24}, 0\}\\
        &\qquad +\max\{\hat{\beta}_{14}-\check{\alpha}_{14}, 0\}\max\{\hat{\beta}_{23}-\check{\alpha}_{23}, 0\}.
        \end{align*}
    \end{proof}
    
    Now, applying Lemma \ref{RCP:res:4thMomentIIDN} and the convergence of $G_n$ in \eqref{RCP:eq:convGn}, elementary calculations yield for $((t,\lambda), (t',\lambda')) \in [0,1]^2 \times [0,1]^2 \setminus \ \{((t, \lambda), (t', \lambda')): t = \lambda = 1 \text{ or } t' = \lambda' = 1\},$
    \begin{align*}
        \E&\left[\tilde{Z}^3_n(t,\lambda) \tilde{Z}_n(t',\lambda')\right]\\
        &\qquad \rightarrow 3\sigma_A^4 \begin{cases}
        (1-\lambda)^3(1-\lambda') \frac{t}{1-t} \min\left\{\frac{t}{1-t}, \frac{t'}{1-t'}\right\}, & \text{ if } t \leq \lambda, t' \leq \lambda'\\
(1-\lambda)^3\lambda' \left(\frac{t}{1-t}\right)^2\min\left\{\frac{1-t}{t}, \frac{1-t'}{t'}\right\}, & \text{ if } t \leq \lambda, t' > \lambda'\\
\lambda^3 (1-\lambda') \frac{1-t}{t} \frac{t'}{1-t'}\min\left\{\frac{1-t}{t}, \frac{1-t'}{t'}\right\}, & \text{ if } t > \lambda, t' \leq \lambda'\\
\lambda^3 \lambda' \frac{1-t}{t}\min\left\{\frac{1-t}{t}, \frac{1-t'}{t'}\right\}, & \text{ if } t > \lambda, t' > \lambda'
        \end{cases}
    \end{align*}
    and
    \begin{align*}
        &\E\left[\tilde{Z}^2_n(t,\lambda) \tilde{Z}^2_n(t',\lambda')\right]\\
        &\quad \rightarrow \sigma_A^4 \begin{cases}
        (1-\lambda)^2(1-\lambda')^2\left\{\frac{tt'}{(1-t)(1-t')} + 2 \left(\min\left\{\frac{t}{1-t}, \frac{t'}{1-t'}\right\}\right)^2\right\}, & \text{ if } t \leq \lambda, t' \leq \lambda'\\
(1-\lambda)^2(\lambda')^2 \left\{\frac{t(1-t')}{(1-t)t'}+2\left(\frac{t}{1-t}\right)^2\left(\min\left\{\frac{1-t}{t}, \frac{1-t'}{t'}\right\}\right)^2\right\}, & \text{ if } t \leq \lambda, t' > \lambda'\\
\lambda^2 (1-\lambda')^2 \left\{\frac{(1-t)t'}{t(1-t')} + 2\left(\frac{t'}{1-t'}\right)^2 \left(\min\left\{\frac{1-t}{t}, \frac{1-t'}{t'}\right\}\right)^2\right\}, & \text{ if } t > \lambda, t' \leq \lambda'\\
\lambda^2 (\lambda')^2 \left\{\frac{(1-t)(1-t')}{tt'} + 2\left(\min\left\{\frac{1-t}{t}, \frac{1-t'}{t'}\right\}\right)^2\right\}, & \text{ if } t > \lambda, t' > \lambda'
        \end{cases}.
    \end{align*}
    Note that for all $((t,\lambda), (t', \lambda')) \in \{((t,\lambda), (t', \lambda')): t = \lambda = 1 \text{ or } t' = \lambda' = 1\}$, we have for all $n\in \N$, 
    \[\E[\tilde{Z}_n^3(t,\lambda) \tilde{Z}_n(t', \lambda')] = \E[\tilde{Z}_n^2(t,\lambda) \tilde{Z}_n(t',\lambda')] = 0.\]
    Hence, the fourth moment of $\hat{Z}_n((t,\lambda),(t',\lambda'))$ satisfies for all $(t,\lambda), (t',\lambda') \in [0,1]^2,$
    \begin{align*}
        \E\left[\hat{Z}^4_n((t,\lambda),(t',\lambda'))\right]&=x^4 \E\left[\tilde{Z}^4_n(t,\lambda)\right] + 4x^3y \E\left[\tilde{Z}^3_n(t,\lambda)\tilde{Z}_n(t',\lambda')\right]\\
        &\qquad + 6 x^2y^2\E\left[\tilde{Z}^2_n(t,\lambda)\tilde{Z}^2_n(t',\lambda')\right]\\
        &\qquad + 4 xy^3\E\left[\tilde{Z}_n(t,\lambda)\tilde{Z}^3_n(t',\lambda')\right] + y^4 \E\left[\tilde{Z}^4_n(t',\lambda')\right]\\
        &\rightarrow 3 \left(\hat{c}((t,\lambda),(t',\lambda'))\right)^2
    \end{align*}
    as desired. An application of the fourth moment theorem \cite[Theorem 1]{NP05} yields that 
    \[\hat{Z}_n((t,\lambda),(t',\lambda')) \Rightarrow x Z^*(t,\lambda) + yZ^*(t',\lambda') =: \hat{Z}^*((t,\lambda),(t',\lambda')),\]
    where $\hat{Z}^*((t, \lambda), (t',\lambda'))$ is a normally distributed random variable with mean zero and variance equal to $\hat{c}((t,\lambda), (t',\lambda')).$ Together with the tightness of the sequence $(\tilde{Z}_n)_{n\in \N}$, the statement of Theorem \ref{RCP:res:limitZn1} follows.
\end{proof}

\begin{Exp}[Normal observations, univariate case $m =1$]$ $
    \begin{enumerate}
        \item Change in the mean when the variance $\sigma^2$ is known and constant: we have $d = 1$ and the density function equals
        \[f(x; \theta) = \exp\left(\frac{x}{\sigma^2}\theta - \frac{\theta^2}{2\sigma^2}-\frac{1}{2}\log(2\pi\sigma^2)\right),\]
        and therefore $T(x) = x/\sigma^2,$ $H(x) = x^2\sigma^2/2,$ and $A(\theta) = \theta^2/(2\sigma^2).$ Elementary calculations yield that $A''(\theta) = 1/\sigma^2,$ $H''(x) = \sigma^2.$
        \item Simultaneous change in the mean and variance: we have that $d = 2$ and for $\mu \in \R$ and $\sigma\in \R_+$, the density is given by
        \[f(x; \mu, \sigma^2) = \exp\left(\frac{\mu}{\sigma^2}x + \frac{1}{2\sigma^2}x^2 - \frac{\mu}{2\sigma^2} - \frac{1}{2}\log(2\pi \sigma^2)\right).\]
        Now, setting $\theta(\mu, \sigma^2) = (\theta_1, \theta_2)^T = (\mu/\sigma^2, -1/(2\sigma^2))^T\in \R^2,$ the density function can be written as
        \[f(x; \theta) = \exp\left(\theta_1 x - \theta_2 x^2 + \frac{1}{4}\frac{\theta^2_1}{\theta_2} - \frac{1}{2}\log\left(-\frac{\pi}{\theta_2}\right)\right),\]
        and therefore $T(x) = (x, x^2)^T,$ $H(x) = -\frac{1}{2}\log(2\pi(x_2-x^2_1)),$ and $A(\theta) = -\frac{1}{4}\frac{\theta^2_1}{\theta_2} + \frac{1}{2}\log(-\pi/\theta_2).$ Elementary calculations yield that 
        \[A''(\theta) = \begin{pmatrix}
                            -\frac{1}{2\theta_2} & \frac{\theta_1}{2\theta_2^2}\\
                            \frac{\theta_1}{2\theta^2_2} & -\frac{\theta^2_1}{2\theta^3_2}+\frac{1}{2\theta_2^2}
                        \end{pmatrix}, \quad 
        H''(x) = \begin{pmatrix}
                        \frac{x^2_1 + x_2}{(x_1^2 - x_2)^2} & -\frac{x_1}{(x^2_1 - x_2)^2}\\
                        -\frac{x_1}{(x_1^2 - x_2)^2} & \frac{1}{2(x^2_1 - x_2)^2}
        \end{pmatrix}.\]
    \end{enumerate}
\end{Exp}

\begin{Exp}[Normal observations, multivariate case]$ $
    \begin{enumerate}
        \item Change in the mean when the covariance matrix $\Sigma$ is known and constant: we have $d = m,$ and the density function equals 
        \[f(x; \theta) = \exp\left(\theta^T\Sigma^{-1}x - \frac{1}{2}\theta^T \Sigma^{-1}\theta - \frac{1}{2}x^T \Sigma^{-1}x - \log((2\pi)^{m/2}det(\Sigma))\right),\]
        where $\Sigma \in \R^{m\times m}$ is symmetric, positive definite. Hence, $T(x) = \Sigma^{-1}x,$ $H(x) = \frac{1}{2}x^T\Sigma x,$ and $A(\theta) = \frac{1}{2}\theta^T\Sigma^{-1}\theta.$ Elementary calculations yield that $H''(x) = \Sigma,$ $A''(\theta) = \Sigma^{-1}.$
        \item Simultaneous change in the mean and covariance matrix: we have $d = m+m^2,$ and for $\mu \in \R^m$ and $\Sigma \in \R^{m\times m}$ being symmetric and positive definite, the density function equals
        \[f(x; \mu, \Sigma) = \exp\left(\mu^T \Sigma^{-1} x - \frac{1}{2}\mu^T\Sigma^{-1}\mu - \frac{1}{2}x^T\Sigma^{-1}x-\log((2\pi)^  {m/2}\det(\Sigma))\right).\]
        With a little abuse of notation, we identify $\R^{m\times m} \equiv \R^{m^2}$ and $\Sigma \equiv vech(\Sigma) \in \R^{m^2}.$ Then, setting $\theta(\mu, \Sigma) = (\theta_1, \theta_2)^T = (\Sigma^{-1}\mu, -\frac{1}{2}\Sigma^{-1}) \in \R^{m+m^2}$ the density can be rewritten in form of its natural parametrization, i.e.,
        \[f(x; \theta) = \exp\left(\theta_1^Tx + x^T\theta_2 x + \frac{1}{4}\theta_1^T \theta^{-1}_2\theta_1 - \frac{1}{2}\log(\det(-\pi \theta^{-1}_2))\right),\]
        and therefore $T(x) = (x, xx^T) \in \R^{m+m^2},$ $H(x) = -\frac{1}{2}\log(\det(2\pi(x_2-x_1x_1^T)))$, for $(x_1, x_2)\in \R^m \times \R^{m^2},$ and $A(\theta) = -\frac{1}{4}\theta^T_1 \theta^{-1}_2 \theta_1 + \frac{1}{2}\log(\det(-\pi\theta_2^{-1})).$ Elementary calculations yield that
        \[A'(\theta) := (A_{\theta_1}, A_{\theta_2})^T := \left(-\frac{1}{2}\theta_1^T\theta_2^{-1},\,  \frac{1}{4}(\theta_1^T\theta_2^{-1}\otimes \theta_1^T\theta_2^{-1}) - \frac{1}{2}vech(\theta_2^{-1})^T\right)^T,\]
        and
        \[A''(\theta) := \begin{pmatrix}
        \frac{\partial A_{\theta_1}}{\partial \theta_1}& \frac{\partial A_{\theta_1}}{\partial \theta_2}\\
        \frac{\partial A_{\theta_2}}{\partial \theta_1}&  \frac{\partial A_{\theta_2}}{\partial \theta_2},
        \end{pmatrix}\]
        where
        $\frac{\partial A_{\theta_1}}{\partial \theta_1}=-\frac{1}{2}\theta^{-1}_2 \in \R^{m\times m},$ $\frac{\partial A_{\theta_1}}{\partial \theta_2} = (\frac{\partial A_{\theta_2}}{\partial \theta_1})^T = (\theta^T_1 \theta_2^{-1})^T \otimes vech(\theta_2^{-1})^T \in \R^{m \times m^2},$ and $\frac{\partial A_{\theta_2}}{\partial \theta_2} = \frac{1}{2}(\theta_1^T\theta_2^{-1}) \otimes (\theta_1^T\theta_2^{-1}) \otimes vech(\theta_2^{-1}) \in \R^{m^2 \times m^2}.$ 
    \end{enumerate}
\end{Exp}

The following corollary shows that under a same rescaling of size $(n\delta^2)^{-1}$ used for $Z_n$, the remainder process $R_n$ converges in probability to the zero process.

\begin{Prop}[A limit theorem for $R_n$ under the alternative]\label{RCP:res:limitRn1}
    Let the assumptions of Theorem \ref{RCP:res:limitZn1} be satisfied. Then, the remainder process $(n\delta^2)^{-1}R_n$ converges in probability in the Skorokhod topology to the zero process. 
\end{Prop}

\begin{proof}
    Observe that by equation \eqref{RCP:eq:testStatH1-centered}, we have for $k\leq k^*$,
    \begin{equation*}
    \begin{split}
        R_n(k,k^*)&= \frac{1}{2k}\left(\sum_{i=1}^k(T(X_{1,i}) -\tau_1)\right)^T H''(\xi_1)\left(\sum_{i = 1}^k(T(X_{1,i}) - \tau_1)\right)\\
        &\qquad + \frac{1}{2(n-k)}\left(\sum_{i = k+1}^{k^*}(T(X_{1,i}) - \tau_1) + \sum_{i = k^*+1}^n(T(X_{2,i})-\tau_2)\right)^T H''(\xi_2)\\
        &\qquad \qquad \times \left(\sum_{i=k+1}^{k^*} (T(X_{1,i}) - \tau_1) + \sum_{i = k^*+1}^n(T(X_{2,i}) - \tau_2)\right)\\
        &\qquad - \frac{1}{2n}\left(\sum_{i=1}^{k^*}(T(X_{1,i}) -\tau_1) + \sum_{i=k^*+1}^n(T(X_ {2,i})-\tau_2)\right)^TH''(\xi)\\
        &\qquad \qquad \times \left(\sum_{i=1}^{k^*}(T(X_{1,i}) - \tau_1) + \sum_{i=k^*+1}^n(T(X_{2,i}) -\tau_2)\right),\\
    \end{split}
    \end{equation*}
    where $\xi_1$  is in the interval connecting $B_n(k,k^*)$ and $\tau_1,$ $\xi_2$ is in the interval connecting $B_n^*(k,k^*)$ and $(k^*-k)/(n-k)\tau_1 + (n-k^*)/(n-k)\tau_2$, and $\xi$ is in the interval connecting $B_n(n,k^*)$ and $k^*/n\tau_1 + (n-k^*)/n\tau_2.$ In contrast, for $k > k^*,$ we have
    \begin{equation*}
    \begin{split}
        R_n(k,k^*)&= \frac{1}{2k}\left(\sum_{i=1}^{k^*}(T(X_{1,i}) -\tau_1) + \sum_{i=k^*+1}^k (T(X_{2,i}) -\tau_2) \right)^T H''(\tilde{\xi}_1)\\
        &\qquad \qquad \times \left(\sum_{i = 1}^{k^*}(T(X_{1,i}) - \tau_1) + \sum_{i=k^*+1}^k(T(X_{2,i}) -\tau_2)\right) \\
        &\qquad + \frac{1}{2(n-k)}\left(\sum_{i = k+1}^{n}(T(X_{2,i})-\tau_2)\right)^T H''(\tilde{\xi}_2)\left(\sum_{i=k+1}^{n} (T(X_{2,i}) - \tau_2)\right)\\
        &\qquad - \frac{1}{2n}\left(\sum_{i=1}^{k^*}(T(X_{1,i}) -\tau_1) + \sum_{i=k^*+1}^n(T(X_{2,i})-\tau_2)\right)^TH''(\xi)\\
        &\qquad \qquad \times \left(\sum_{i=1}^{k^*}(T(X_{1,i}) - \tau_1) + \sum_{i=k^*+1}^n(T(X_{2,i}) -\tau_2)\right)^T,
    \end{split}
    \end{equation*}
    where $\tilde{\xi}_1$ is in the interval connecting $B_n(k,k^*)$ and $k^*/k\tau_1 + (k-k^*)/k \tau_2$ and $\tilde{\xi}_2$ is in the interval connecting $B_n^*(k,k^*)$ and $\tau_2.$ Now, an application of Donsker's theorem (similarly to Lemma \ref{RCP:res:clt1}) together with Assumption \ref{RCP:ass:cases} yields that the sequence $(R_n)_{n\in \N}$ is tight, where for $t,\lambda \in [0,1],$
    \[R_n(t,\lambda) := \sum_{k=1}^{n-1} \sum_{k^*=1}^{n-1} R_n(k,k^*) \1_{\{nt\in [k, k+1)\}} \1_{\{n\lambda \in [k^*, k^*+1)\}}.\]
    Moreover, by Assumption \ref{RCP:ass:reg} and \ref{RCP:ass:cases} (in particular, since $A$ has continuous derivatives up to the second order), we conclude that $n\delta^2 \rightarrow \infty$ as $n\rightarrow \infty.$ Hence, 
    \[\frac{1}{\sqrt{n\delta^2}}R_n \rightarrow 0 \quad \text{ in probability}\]
    in the Skorokhod topology as $n\rightarrow \infty.$ This finishes the proof.
\end{proof}

\begin{Cor}\label{RCP:res:ZnRn}
    Let the assumptions of Theorem \ref{RCP:res:limitZn1} be satisfied. Then, we have
    \[\sup_{t,\lambda \in [0,1]}|Z_n(t,\lambda) + R_n(t,\lambda)| = \mathcal{O}_{\Pro}\left(\sqrt{n\delta^2}\right),\]
\end{Cor}

\begin{proof}
    Recall, the definition of the piecewise constant interpolations $Z_n := (Z_n(t,\lambda))_{t, \lambda \in [0,1]}$ and $(R_n(t,\lambda))_{t,\lambda \in [0,1]}$ of $(Z_n(k,k^*): k,k^* \in \{1,\cdots, n-1\})$ and $(R_n(k,k^*): k,k^* \in \{1,\cdots, n-1\}).$ Then, studying the proof of Theorem \ref{RCP:res:limitZn1} and Proposition \ref{RCP:res:limitRn1}, we conclude the tightness of the sequences $(\tilde{Z}_n)_{n\in \N} := ((n\delta^2)^{-1/2} Z_n)_{n\in \N}$ and $(R_n)_{n\in \N}.$ Hence,
    \[\frac{1}{\sqrt{n\delta^2}}\sup_{t,\lambda \in [0,1]}|Z_n(t,\lambda) + R_n(t,\lambda)| = \mathcal{O}_{\Pro}(1),\]
    which yields the stated claim.
\end{proof}

Once the null hypothesis $H_0$ ``no change point'' is rejected, one is interested in locating the change point $k^*_n$ or the change point fraction $\lambda^*_n := k^*_n/n.$ For this, we suggest the estimator
\begin{equation}\label{RCP:def:estCP}
    \hat{\lambda}_n := \frac{1}{n}\argmax_{1\leq k \leq n-1}\{2S_n(k)\}.
\end{equation}

In the following, we denote by $\hat{k}_n := n\hat{\lambda}_n$ the estimator of the location of the change point $k^*_n.$ Note that, as always, the estimator in \eqref{RCP:def:estCP} is random, as it depends on the data $X_1,\cdots, X_n$ and therefore, also on the true location of the change point $k^*_n.$ In the following, we will study the properties of this estimator. 

\begin{The}[Consistency of $\hat{\lambda}_n$]\label{RCP:res:consistency1}
    Let the assumptions of Theorem \ref{RCP:res:limitZn1} hold. Then, 
    \[\delta^2|\hat{k}_n - k^*_n| = \mathcal{O}_{\Pro}(1).\]
    In particular, $\hat{\lambda}_n$ is a consistent estimator of $\lambda^*_n$ with $|\hat{\lambda}_n - \lambda^*_n| = \mathcal{O}_{\Pro}((\delta^2 n)^{-1}).$
\end{The}

\begin{Rmk}[Convergence rates and minimum detectable size in slightly different models]
    Csörg\H{o} and Horváth \cite{CH97} also studied the consistency of the estimator in \eqref{RCP:def:estCP} provided that the true location of the change point $k^*_n$ is deterministic under slightly different assumptions:
    \begin{itemize}
        \item[i)] If the size of the change point is independent of $n$, they still obtain a convergence rate for the estimator $\hat{\lambda}_n$ of order $n^{-1}.$
        \item[ii)] If the change point fraction satisfies $k^*_n/n\rightarrow 0$ as $n\rightarrow \infty,$ i.e., the data contain an early change point, they obtain the same convergence rate as in Theorem \ref{RCP:res:consistency1}, but the detectable size of the change point has to be generally of larger order satisfying
        \[k^*_n \Delta^2 \rightarrow \infty\]
    \end{itemize}
    Although, we do not study these cases in our work, we expect to obtain similar results assuming that $k^*_n$ is itself random.
\end{Rmk}

\begin{proof}
    First, observe that for all $k,k^* \in \{1,\cdots, n-1\}$ with $k\leq k^*,$ we have
    \begin{equation*}
    \begin{split}
        \mu_n&(k,k^*)-\mu_n(k^*,k^*)\\
        &= (k-k^*) H(\tau_1) + (n-k)H\left(\frac{k^*-k}{n-k}\tau_1 + \frac{n-k^*}{n-k}\tau_2\right) - (n-k^*) H(\tau_2)\\
        &= (k^*-k) \left(H\left(\frac{k^*-k}{n-k}\tau_1 + \frac{n-k^*}{n-k}\tau_2\right) - H(\tau_1)\right)\\
        &\hspace{3cm}+ (n-k^*) \left(H\left(\frac{k^*-k}{n-k}\tau_1 +\frac{n-k^*}{n-k}\tau_2\right) - H(\tau_2)\right).
    \end{split}
    \end{equation*}
    Now, two applications of Taylor's formula of the second order yield that
    \begin{equation*}
    \begin{split}
        &\mu_n(k,k^*) - \mu_n(k^*,k^*)= \frac{1}{2}\frac{(k^*-k)(n-k^*)}{n-k}\Bigg(2(\tau_2-\tau_1)^T\left(H'(\tau_1)-H'(\tau_2)\right)\\
        & +\frac{n-k^*}{n-k}(\tau_1-\tau_2)^TH''(\tau_1) (\tau_1 - \tau_2)+ \frac{k^*-k}{n-k}(\tau_1-\tau_2)^TH''(\tau_2)(\tau_1-\tau_2)\Bigg)\\
        & + o\left(\frac{(k^*-k)(n-k^*)\delta^2}{n-k}\right),
    \end{split}
    \end{equation*}
    where we obtain the last summand by bounding the Lagrange remainder term. By the mean value theorem, there exists $\xi$ in the interval connection $\tau_1$ and $\tau_2$ such that 
    \begin{equation}\label{RCP:eq:mun}
    \begin{split}
        &\mu_n(k,k^*) - \mu_n(k^*, k^*) = \frac{1}{2}\frac{(k^*-k)(n-k^*)}{n-k}\Bigg(- 2 (\tau_1 - \tau_2)^T H''(\xi) (\tau_1 - \tau_2)\\
        & + \frac{n-k^*}{n-k} (\tau_1 -\tau_2)^T H''(\tau_1) (\tau_1 - \tau_2) + \frac{k^*-k}{n-k}(\tau_1 -\tau_2)^T H''(\tau_2) (\tau_1 - \tau_2)\Bigg)\\
        & + o\left(\frac{(k^*-k)(n-k^*) \delta^2}{n-k}\right).
    \end{split}
    \end{equation}
    Now, since $(\tau_1, \tau_2) \rightarrow (\tau_A, \tau_A)$ as $n\rightarrow \infty$ thanks to Assumption \ref{RCP:ass:cases} ii), we conclude for $n$ large enough, that the above difference is negative, increasing in $k$ for fixed $k^*,$ and
    \[\mu_n(k,k^*) - \mu_n(k^*,k^*) = \mathcal{O}\left(\frac{(k^*-k)(n-k^*)\delta^2}{n-k}\right).\]
    In particular, for all $\lambda \in (0,1),$ $t\in [0,\lambda - \kappa/(\delta^2 n))$ for some $\kappa > 0$, and $n$ large enough, we have
    \begin{equation}\label{RCP:eq:rateMu}
    \mu_n(t, \lambda) - \mu_n(\lambda,\lambda) = \mathcal{O}\left(n\delta^2(1-\lambda)\right),
    \end{equation}
    where we used that the above difference is increasing in $t$ for fixed $\lambda.$
    Moreover, observe that
   \[S_n(t, \lambda) = \mu_n(t,\lambda) + Z_n(t,\lambda) + R_n(t,\lambda)\]
   and thanks to Corollary \ref{RCP:res:ZnRn}, for any null sequence $(a_n)_{n\in \N}$ and $\epsilon > 0,$ we have
   \begin{equation}\label{RCP:eq:ZnRnVanish}
   \Pro\left[\frac{a_n}{\sqrt{n\delta^2}}\sup_{t,\lambda \in [0,1]} |Z_n(t,\lambda) + R_n(t,\lambda)| > \epsilon\right] \rightarrow 0.
   \end{equation}
    In the following, let us denote by $V_n(t,\lambda) := Z_n(t,\lambda) + R_n(t,\lambda).$ By Assumption \ref{RCP:ass:cases}, $\lambda^* \in [\gamma, 1-\gamma]$ for some $\gamma \in (0,1/2)$ with probability one. Hence, for each $\epsilon > 0,$ there exists an $N\in \N$ such that for each $n\geq N$, we have
    \[\Pro\left[\lambda^*_n \notin [\gamma/2, 1- \gamma/2]\right] < \epsilon.\] Next, let us define the set $A_n := \{\hat{k}_n \leq k^*_n\} \cap \{\lambda^*_n \in [\gamma/2, 1-\gamma/2]\}$. Applying \eqref{RCP:eq:rateMu} and \eqref{RCP:eq:ZnRnVanish}, for each $\kappa > 0,$ we have 
    \begin{equation*}
    \begin{split}
        \Pro\left[k^*_n - \hat{k}_n > \frac{\kappa}{\delta^2},\,  A_n \right] &= \Pro\left[\lambda^*_n - \hat{\lambda}_n >\frac{\kappa}{n\delta^2}, \, A_n \right]\\
        &= \Pro\left[\sup_{t\in [0,\lambda^*_n - \kappa/(n \delta^2))} S_n(t, \lambda^*_n) - S_n(\lambda^*_n, \lambda^*_n) \geq 0, \, A_n \right]\\
        &\leq \Pro\left[\sup_{t\in [0, \lambda^*_n-\kappa/(n\delta^2))} V_n(t, \lambda^*_n) - V_n(\lambda^*_n, \lambda^*_n) \geq n\delta^2(1-\lambda^*_n), \, A_n \right]\\
        &\leq \Pro\left[\sup_{\lambda \in [\gamma/2,1-\gamma/2]}\sup_{t \in [0, \lambda - \kappa/(n \delta^2))} 4 \gamma^{-1} |V_n(t, \lambda)| \geq n\delta^2, \, A_n\right] \rightarrow 0. 
    \end{split}
    \end{equation*}
    Analogously, we can show for each $\kappa > 0,$ that
    \[\Pro\left[\hat{k}_n - k^*_n > \frac{\kappa}{\delta^2},  \, B_n\right] \rightarrow 0,\]
    where $B_n := \{\hat{k}_n > k^*_n\} \cap \{\lambda^*_n \in [\gamma/2, 1-\gamma/2]\}.$
    Finally, for each $\epsilon > 0,$ there exist an $N\in \N$ and $\kappa > 0$ such that for all $n \geq N,$ we have 
    \begin{align*}
        \Pro &\left[|k^*_n - \hat{k}_n| > \frac{\kappa}{\delta^2}\right]\\
         &\leq \Pro\left[|k^*_n - \hat{k}_n| > \frac{\kappa}{\delta^2}, \, \lambda^*_n \in [\gamma/2, 1-\gamma/2] \right] + \Pro\left[\lambda^*_n \notin [\gamma/2, 1-\gamma/2]\right]\\
        & = \Pro\left[k^*_n - \hat{k}_n > \frac{\kappa}{\delta^2}, \,A_n \right] + \Pro\left[\hat{k}_n - k^*_n > \frac{\kappa}{\delta^2}, \,B_n \right] + \Pro\left[\lambda^*_n \notin [\gamma/2, 1-\gamma/2]\right]\\
        &\leq 3\epsilon.
    \end{align*}
    Hence, $\delta^2|k^*_n - \hat{k}_n| = \mathcal{O}_{\Pro}(1).$ Finally, since
    \[\delta^2|k^*_n - \hat{k}_n| = n\delta^2|\lambda^*_n - \hat{\lambda}_n|,\]
    we conclude that $\hat{\lambda}_n$ is a consistent estimator of $\lambda^*_n$ and $|\lambda_n^* - \hat{\lambda}_n| = \mathcal{O}_{\Pro}((n\delta^2)^{-1}).$
\end{proof}

In order to construct confidence intervals for $k^*_n,$ we need to establish limit distributions for the deviation $\delta^2(\hat{k}_n - k^*_n)$ provided that $H_1$ holds true. The next theorem gives us a first limit result for this deviation. However, it is only of theoretical interest, since $(\tau_1 - \tau_2)^T H''(\tau_A)(\tau_1-\tau_2)$, the size of a change, is unknown.

\begin{The}[Limit distribution of $\hat{k}_n$ under the alternative]\label{RCP:res:limitDist}
    Let the assumptions of Theorem \ref{RCP:res:limitZn1} be satisfied and assume that $\sigma^2_A$ introduced in Theorem \ref{RCP:res:limitZn1} exists. Moreover, let us introduce the process
    \[W^*(t) := \begin{cases}
        \sigma_AW_1(-t) - \frac{1}{2}\sigma_A^2|t|, & t <0,\\
        0, & t = 0,\\
        \sigma_AW_2(t) - \frac{1}{2}\sigma_A^2|t|, & t > 0,
    \end{cases}\]
    where $W_1$ and $W_2$ are two independent Brownian motions. Then,
    \[\delta^2(\hat{k}_n-k^*_n) \Rightarrow \argmax_{u\in (-\infty, \infty)} W^*(u).\]
\end{The}

Note that $\argmax_{u\in (-\infty, \infty)} W^*(u)$ is the canonical limit distribution from the literature (cf. e.g. \cite{CH97, JWY18}) for the deviation $\delta^2(\hat{k}_n - k^*_n)$ provided that the size of the change point vanishes as $n \rightarrow \infty.$ Moreover, the law of the iterated logarithm for the Brownian motion implies the almost sure finiteness of $\argmax_{u\in (-\infty, \infty)} W^*(u).$

We state the proof of this theorem at the end of this section. With a slight rescaling of the left hand-side in Theorem \ref{RCP:res:limitDist}, we can establish a distribution-free limit process.

\begin{Cor}\label{RCP:res:limitDistC1}
    Let the assumptions of Theorem \ref{RCP:res:limitZn1} be satisfied. Then, we have
    \[(\tau_1-\tau_2)^T H''(\tau_A)(\tau_1-\tau_2)(\hat{k}_n - k^*_n) \Rightarrow  \argmax_{u\in (-\infty, \infty)} \hat{W}(u),\]
     where the limit process $\hat{W}$ is defined by 
    \[\hat{W}(t) := \begin{cases}
        W_1(-t) - \frac{1}{2}|t|, & t <0,\\
        0, & t = 0,\\
        W_2(t) - \frac{1}{2}|t|, & t > 0,
    \end{cases}\]
    for two independent Brownian motions $W_1$ and $W_2$.
\end{Cor}

\begin{proof}
    By the scaling property of the Brownian motion, i.e., $W(t) = c^{-1/2}W(ct),$ for a Brownian motion $W = (W(t))_{t \in [0,1]},$ $t\in [0,1]$, and all $c \in \R,$ we conclude that
    \begin{align*}
        \argmax_{u\in (-\infty, \infty)} W^*(u) \quad \text{ and } \quad \frac{1}{\sigma_A^2} \argmax_{u\in (-\infty, \infty)}\hat{W}(s)
    \end{align*}
    have the same distribution, where
    \[\sigma^2_A := \lim_{n\rightarrow \infty} \frac{(\tau_1-\tau_2)^TH''(\tau_A)(\tau_1-\tau_2)}{\delta^2}.\]
    This finishes the proof.
\end{proof}

Even if the right hand-side in Corollary \ref{RCP:res:limitDistC1} is distribution-free, $(\tau_1 - \tau_2)^TH''(\tau_A)(\tau_1-\tau_2),$ the size of a change, occurring on the left hand-side, is still unknown. Hence, in order to be able to construct confidence intervals for $k^*_n,$ we need to estimate $(\tau_1 - \tau_2)^TH''(\tau_A)(\tau_1-\tau_2),$ the size of a change. For that, we use the estimator 
\begin{equation}
    (B_n(\hat{k}_n) - B^*_n(\hat{k}_n))^T H''(B_n(n))(B_n(\hat{k}_n) - B^*_n(\hat{k}_n))
\end{equation}

For this reason, we will show in the next lemma that this is indeed a consistent estimator for the size of a change.

\begin{Lem} \label{RCP:res:constEstSize}
    Under the assumptions of Theorem \ref{RCP:res:limitZn1}, we have
    \[\frac{(B_n(\hat{k}_n) - B^*_n(\hat{k}_n))^TH''(B_n(n))(B_n(\hat{k}_n)-B^*_n(\hat{k}_n))}{(\tau_1-\tau_2)^TH''(\tau_A)(\tau_1-\tau_2)} \rightarrow 1 \quad \text{ in probability.}\]
\end{Lem}

\begin{proof}
    Since $\E[T(X_{1,1})] = A'(\theta^{(1)}_0) = \tau_1$ and $\E[T(X_{2,1})] = A'(\theta^{(2)}_0) = \tau_2,$ we conclude by Remark \ref{RCP:rmk:derivativesNormal}, Assumption \ref{RCP:ass:cases}, and an application of the weak law of large numbers that
    \[B_n(n) = \frac{1}{n}\sum_{i=1}^n T(X_i) \rightarrow A'(\theta_A) =: \tau_A \quad \text{ in probability.}\]
    Again, by Remark \ref{RCP:rmk:derivativesNormal}, we conclude that $H$ has continuous derivatives up to the third order. Therefore, we also have
    \[H''(B_n(n)) \rightarrow H''(\tau_A) \quad \text{ in probability.}\]
    Moreover, again by the weak law of large numbers and the consistency property of $\hat{k}_n$ (cf. Theorem \ref{RCP:res:consistency1}), we get that
    \begin{align*}
        &\left|B_n(\hat{k}_n)-\tau_1\right|=\left|\frac{1}{\hat{k}_n} \left(\sum_{i=1}^{\hat{k}_n} \left(T(X_{1,i}) - \tau_1\right)\right) \right|\1_{\{\hat{k}_n \leq k^*_n\}}\\
        &\qquad + \left|\frac{1}{\hat{k}_n}\left(\sum_{i=1}^{k^*_n}(T(X_{1,i}) - \tau_1) + \sum_{i=k^*_n}^{\hat{k}_n}(T(X_{2,i})-\tau_2) - (\hat{k}_n - k^*_n)(\tau_1 - \tau_2)\right)\right| \1_{\{\hat{k}_n > k^*_n\}}\\
        &= o_{\Pro}(1).
    \end{align*}
    Similarly, we conclude that
    \[\left|B_n^*(\hat{k}_n) - \tau_2\right|  = o_{\Pro}(1).\]
    Combining the above observations,  we finally obtain the stated result
    \[\frac{(B_n(\hat{k}_n) - B^*_n(\hat{k}_n))^TH''(B_n(n))(B_n(\hat{k}_n)-B^*_n(\hat{k}_n))}{(\tau_1-\tau_2)^TH''(\tau_A)(\tau_1-\tau_2)} \rightarrow 1 \quad \text{ in probability.}\]
\end{proof}

With all these preparations done, we are finally ready to state a distribution-free limit theorem for the deviation $(\hat{k}_n - k^*_n)$ under the alternative. This result can be used to build confidence intervals for $k^*_n.$

\begin{Cor}\label{RCP:res:limitDistC2}
    Let the assumptions of Theorem \ref{RCP:res:limitZn1} be satisfied. Then, we have
    \[\left(B_n(\hat{k}_n) - B^*_n(\hat{k}_n)\right)^TH''(B_n(n))\left(B_n(\hat{k}_n)-B^*_n(\hat{k}_n)\right)\left(\hat{k}_n - k^*_n\right) \Rightarrow \argmax_{u \in (-\infty, \infty)}\hat{W}(u).\]
\end{Cor}
 
\begin{proof}
    This follows directly from Corollary \ref{RCP:res:limitDistC1}, Lemma \ref{RCP:res:constEstSize}, and an application of Slutzky's Lemma.
\end{proof}

Finally, we finish this section with stating the proof of Theorem \ref{RCP:res:limitDist}.

\begin{proof}[Proof of Theorem \ref{RCP:res:limitDist}]
    We show that for $\kappa > 0,$ it holds
    \[S_n(\lambda + \cdot/(n\delta^2), \lambda) - S_n(\lambda, \lambda) \Rightarrow W^*\]
    in the Skorokhod topology on the space $D([-\kappa, \kappa], \R)$ uniformly over $\lambda \in [\gamma, 1-\gamma].$ First, let us consider $k\in [k^*-\kappa/\delta^2, k^*],$ where $k^*/n \in [\gamma, 1-\gamma].$ Remark \ref{RCP:rmk:derivativesNormal} and a Taylor expansion of the second order yield for all $1\leq k \leq k^*,$
    \begin{align*}
        S_n(k,k^*)& - S_n(k^*,k^*) - (\mu_n(k,k^*) - \mu_n(k^*,k^*))\\
        &= V_{1,n}(k,k^*) + V_{2,n}(k,k^*) + V_{3,n}(k,k^*) + V_{4,n}(k,k^*) + R_{1,n}(k,k^*),
    \end{align*}
    where
    \[V_{1,n}(k,k^*) := H'(\tau_1)^T n^{1/2}\left(W^{(n)}_{1,k} - W^{(n)}_{1,k^*}\right),\]
    \[V_{2,n}(k,k^*) := \frac{n}{2k}\left(W^{(n)}_{1,k}\right)^TH''(\tau_1) W_{1,k}^{(n)} - \frac{n}{2k^*}\left(W^{(n)}_{1,k^*}\right)^T H''(\tau_1) W_{1,k^*}^{(n)},\]
    \begin{align*}
        V_{3,n}(k,k^*) &:= H'\left(\frac{k^*-k}{n-k}\tau_1 + \frac{n-k^*}{n-k}\tau_2\right)^T n^{1/2}\left(W^{(n)}_{1,k^*}-W^{(n)}_{1,k}+ W^{(n)}_{2,n} - W^{(n)}_{2,k^*}\right) \\
        &\qquad - H'(\tau_2)^T n^{1/2}\left(W_{2,n}^{(n)}-W^{(n)}_{2,k^*}\right),
    \end{align*}
    and 
    \begin{align*}
        V_{4,n}(k,k^*) &:= \frac{n}{2(n-k)}\left(W^{(n)}_{1,k^*}-W^{(n)}_{1,k}+ W^{(n)}_{2,n} - W^{(n)}_{2,k^*}\right)^T\\
        &\qquad \qquad \times H''\left(\frac{k^*-k}{n-k}\tau_1 + \frac{n-k^*}{n-k}\tau_2\right) \left(W_{1,k^*}^{(n)}- W^{(n)}_{1,k}+ W^{(n)}_{2,n}- W^{(n)}_{2,k^*}\right)\\
        &\qquad - \frac{n}{2(n-k^*)}\left(W^{(n)}_{2,n} - W^{(n)}_{2,k^*}\right)^T H''(\tau_2) \left(W^{(n)}_{2,n}- W^{(n)}_{2,k^*}\right).
    \end{align*}
    Moreover, Remark \ref{RCP:rmk:derivativesNormal} and Assumption \ref{RCP:ass:cases} imply that the remainder term $R_{1,n}(k,k^*)$ of Lagrange form is of order $o_{\Pro}(1)$ uniformly in $k, k^*$ and hence vanishes in probability as $n\rightarrow \infty.$
    Applying Donsker's theorem (cf. Lemma \ref{RCP:res:clt1}), we obtain for all $\kappa > 0,$
    \begin{equation}\label{RCP:eq:1}
    \begin{split}
        &\sup_{k^*-\kappa/\delta^2 \leq k \leq k^*}\left\|n^{1/2}\delta \left(W^{(n)}_{1,k^*}- W_{1,k}^{(n)}\right)\right\|= \sup_{u\in [-\kappa, 0]}\left\|\delta\hspace{-0.2cm} \sum_{j=k^*+\lfloor u/\delta^2\rfloor}^{k^*}(T(X_{1,i})-\tau_1)\right\|\\
        &= \sup_{u\in [-\kappa, 0]}\left\|\delta \sum_{j=0}^{\lfloor -u/\delta^2\rfloor}(T(X_{1,k^*-j})-\tau_1)\right\|= \mathcal{O}_{\Pro}(1),
    \end{split}
    \end{equation}
    uniformly over $k^* \in [n\gamma, n(1-\gamma)]$.
    Moreover, for $k \in [k^* - \kappa/\delta^2, k^*],$ we have
    \begin{equation}\label{RCP:eq:2}
    \frac{n(k^*-k)}{kk^*} = \mathcal{O}\left((n\delta^2)^{-1}\right).
    \end{equation}
    Hence, applying Assumption \ref{RCP:ass:reg}, Assumption \ref{RCP:ass:cases}, equation \eqref{RCP:eq:1},  equation \eqref{RCP:eq:2}, and Donsker's theorem (cf. Lemma \ref{RCP:res:clt1}), we obtain
    \begin{align*}
        \sup_{k^*-\kappa/\delta^2 \leq k \leq k^*}&\left|V_{2,n}(k,k^*)\right| \leq \sup_{k^*-\kappa/\delta^2 \leq k \leq k^*}\frac{1}{2}\frac{n(k^*-k)}{kk^*}\left|(W^{(n)}_{1,k})^TH''(\tau_1) W^{(n)}_{1,k}\right|\\
        & + \sup_{k^*-\kappa/\delta^2 \leq k \leq k^*}\frac{1}{2}\frac{n}{k^*}\left|\left\{W^{(n)}_{1,k} - W^{(n)}_{1,k^*}\right\}^T H''(\tau_1) W^{(n)}_{1,k}\right|\\
        & + \sup_{k^*-\kappa/\delta \leq k \leq k^*}\frac{1}{2}\frac{n}{k^*}\left|(W^{(n)}_{1,k^*})^T H''(\tau_1)\left\{W^{(n)}_{1,k}-W^{(n)}_{1,k^*}\right\}\right|\\
        &= o_{\Pro}(1)
    \end{align*}
    uniformly over $k^* \in [n\gamma, n(1-\gamma)].$ Bounding $V_{4,n}$ in a similar way, we obtain for arbitrary $\kappa > 0$ and $n$ large enough
    \[\sup_{\lambda \in [\gamma, 1-\gamma]} \sup_{t \in [\lambda-\kappa/(n\delta^2), \lambda]} |V_{2,n}(t,\lambda) + V_{4,n}(t,\lambda)| = o_{\Pro}(1).\]
    Next, applying Donsker's theorem (cf. Lemma \ref{RCP:res:clt1}), we obtain again for $\kappa > 0$ and $n$ large enough
    \begin{align*}
        &\sup_{k^*-\kappa/\delta^2 \leq k \leq k^*}\left|V_{1,n}(k,k^*) + V_{3,n}(k,k^*) - (H'(\tau_2) - H'(\tau_1))^Tn^{1/2}(W^{(n)}_{1,k^*}-W^{(n)}_{1,k})\right|\\
        &= \sup_{k^*-\kappa/\delta^2 \leq k \leq k^*}\Bigg|n^{1/2}\left(H'\left(\frac{k^*-k}{n-k}\tau_1 + \frac{n-k^*}{n-k}\tau_2\right) - H'(\tau_2)\right)^T\\
        &\hspace{6cm} \times \left(W_{1,k^*}^{(n)}-W_{1,k}^{(n)}+ W_{2,n}^{(n)}-W^{(n)}_{2,k^*}\right)\Bigg|\\
        &= \sup_{k^*-\kappa/\delta^2 \leq k\leq k^*} n^{1/2}\delta \frac{k^*-k}{n-k}\Bigg|\frac{(\tau_1 - \tau_2)^T}{\delta}H''(\tau_A)\\
        &\hspace{6cm}\times \left(W^{(n)}_{1,k^*}-W^{(n)}_{1,k}+ W^{(n)}_{2,n}-W^{(n)}_{2,k^*}\right)\Bigg| + o_{\Pro}(1)\\
        &= o_{\Pro}(1)
    \end{align*}
    uniformly over $k^* \in [n\gamma, n(1-\gamma)].$ Combining the above bounds, we conclude that
    \begin{align*}
        \sup_{k^*-\kappa/\delta^2 \leq k \leq k^*}\Big|S_n(k,k^*)& - S_n(k^*,k^*) - (\mu_n(k,k^*) - \mu_n(k^*,k^*))\\
        & - n^{1/2}(H'(\tau_2)- H'(\tau_1))^T(W_{1,k^*}^{(n)}-W_{1,k}^{(n)})\Big| = o_{\Pro}(1)
    \end{align*}
    uniformly over $k^* \in [n\gamma, n(1-\gamma)].$ Similarly, we can show that
    \begin{align*}
        \sup_{k^*\leq k \leq k^*+\kappa/\delta^2} \Big|S_n(k,k^*)& - S_n(k^*,k^*) - (\mu_n(k,k^*) - \mu_n(k^*,k^*))\\
        &- n^{1/2}(H'(\tau_1) - H'(\tau_2))^T(W_{2,k}^{(n)}-W_{2,k^*}^{(n)})\Big| = o_{\Pro}(1)
    \end{align*}
    uniformly over $k^* \in [n\gamma, n(1-\gamma)].$ Furthermore, studying equation \eqref{RCP:eq:mun}, for arbitrary $\kappa >0$, we have
    \[\sup_{-\kappa \leq u \leq \kappa}\left|\mu_n(k^*+u/\delta^2, k^*) - \mu_n(k^*,k^*) - \frac{1}{2}|u|(\tau_1-\tau_2)^T H''(\tau_A)(\tau_1 -\tau_2)/\delta^2\right| = o(1)\]
    uniformly over $k^* \in [n\gamma, n(1-\gamma)].$ Combining the above equation with equation \eqref{RCP:eq:1} and Donsker's theorem (cf. Lemma \ref{RCP:res:clt1}), we finally obtain for arbitrary $\kappa > 0$ that
    \[S_n(\lambda + \cdot/(\delta^2n), \lambda) - S_n(\lambda, \lambda) \Rightarrow W^*\]
    in the Skorokhod topology on the space $D([-\kappa, \kappa], \R)$ uniformly over $\lambda \in [\gamma, 1-\gamma].$ Now, since
    \[\delta^2|\hat{k}_n - k^*_n| = \argmax_{k \in [-\kappa/\delta^2, \kappa/\delta^2]} \left\{S_n(k,k^*) - S_n(k^*,k^*)\right\}\]
    for some $\kappa >0$ large enough and the $\argmax$ function is continuous, we conclude the stated result.
\end{proof}

\section{Simulation study}

In this section, we discuss the derived asymptotic properties for the test statistic $\mathcal{S}_n := \max_{1\leq k \leq n}\{2S_n(k)\}$ and the estimator $\hat{k}_n$ of $k^*_n$ through several simulation studies. \par 
We simulate time series data as follows: the data points $\mathcal{Y}^{(n)} := \{Y^{(n)}_k: k = 1,\cdots, n\}$ are simulated from two independent normal distributions $N_1 \sim \mathcal{N}(n^{-1/2} \mu_1, \sigma_1^2)$ and $N_2 \sim \mathcal{N}(n^{-1/2}\mu_2, \sigma_2^{2})$ such that
\[Y^{(n)}_k \sim N_1 \1_{\{k\leq k^*_n\}} + N_2 \1_{\{k > k^*_n\}}\]
for $k = 1,\cdots, n,$ $\mu_1, \mu_2 \in \R,$ and $\sigma_1, \sigma_2 \in \R_+.$ Then, for $n$ large enough, the discrete-time process $X^{(n)}(t) = \sum_{k=1}^n X^{(n)}_k \1_{\{nt \in [k, k+1)\}}$ with $X^{(n)}_k := n^{-1/n} \sum_{j=1}^k Y^{(n)}_j$ can be approximated by the continuous-time diffusion process
\begin{equation}\label{RCP:eq:limitDiff}
X(t) = \int_0^t \left(\mu_1 \1_ {\{t \leq \lambda^*\}} + \mu_2 \1_{\{t> \lambda^*\}} \right)dt + \int_0^t \left(\sigma_1 \1_{\{t\leq \lambda^*\}} + \sigma_2 \1_{\{t > \lambda^*\}}\right) dW(t),
\end{equation}
where $W$ is a standard Brownian motion and $\lambda^* \in [\gamma, 1-\gamma],$ $\gamma \in (0,1/2),$ is the weak limit of the true change point fraction $k^*_n/n$ (cf. Assumption \ref{RCP:ass:cases}). In this setting, the data points $\mathcal{Y}^{(n)}$ might be interpreted as the scaled increments of $X$ recorded at discrete, equidistant time steps $t^{(n)}_k := k/n,$ $k = 1,\cdots, n.$ Moreover, we simulate the location of the change point by a stopping time that depends on the data. Throughout, the change point fraction $\lambda^*_n := k^*_n/n$ is generated, for $\kappa \in \R$ fixed, by
\begin{equation}\label{RCP:eq:stopTime}
    \lambda^*_n = \inf\left\{t \geq \gamma: X^{(n)}(t) < \kappa\right\} \wedge (1-\gamma).
\end{equation}

\begin{Rmk}
    We present the empirical results where the location of a change point is generated from the stopping time above. Moreover, we also run simulations when $\lambda^*_n$ is sampled from a uniform distribution on $[\gamma, 1-\gamma]$ or a truncated normal distribution with mean $\frac{1}{2}$ and volatility $\frac{1}{6}-\frac{\gamma}{3}$. In both cases, the empirical observations are not significantly different from those we discuss below.
\end{Rmk}

In the following, we choose $\gamma = 0.1,$ $\kappa = -1$ and simulate $n = 10,000$ time steps. All depicted empirical distributions are generated from $m = 10,000$ Monte Carlo runs.\newline

\textbf{Parametric change point detection in the mean and volatility:} Let us choose $\mu_1 = \mu_2 = -2,$ $\sigma_1 = 1,$ and $\sigma_2 = 1.1,$ i.e., we consider a jump in the volatility of $N_1$ versus $N_2$ of size $0.1.$ Then we have $n\Delta^2 = n(\sigma_1 - \sigma_2)^2 = 100.$ So, we might be in the studied setting of Theorem \ref{RCP:res:consistency1} and hope to detect the quite small jump in the volatility. In Figure \ref{RCP:fig:jumpVol}, we depict one realization of $X^{(n)}$ under the alternative and the empirical values of the test statistic $\mathcal{S}_n^{1/2}$ under the null and under the alternative hypothesis. While the jump in volatility is not visible to the naked eye, the empirical values of the test statistic show that our test can very well separate the null hypothesis ``no change point'' from the alternative hypothesis ``there exists one change point''.

\begin{figure}[H]
    \centering
    \small{\textit{Change point model with parameters $\sigma_1 = 1$ and $\sigma_2 = 1.1$}}\\
    \includegraphics[scale = 0.43]{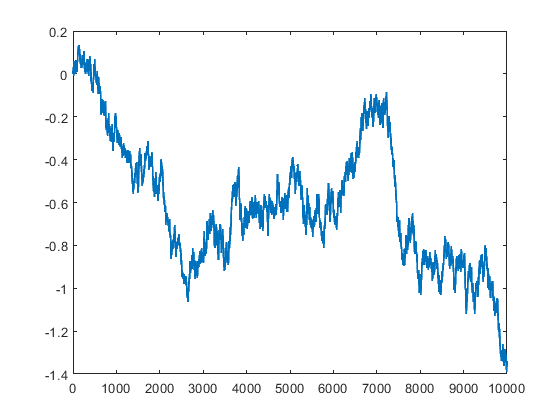}\quad 
    \includegraphics[scale = 0.43]{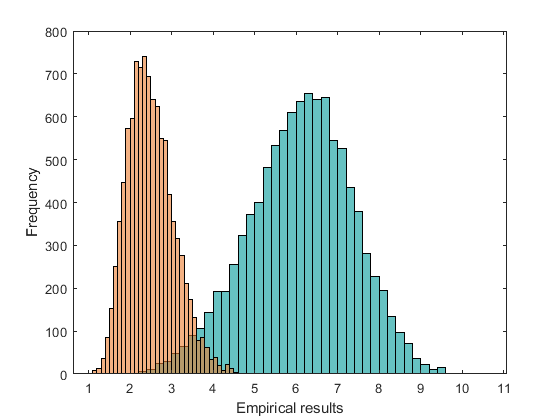}
    \caption{Left: One path of $X^{(n)}$ under the alternative with change point after $n = 2620$ time steps. Right: Empirical values of $\mathcal{S}^{1/2}_n$ under the null (orange) and the alternative (turquoise).}
    \label{RCP:fig:jumpVol}
\end{figure}

In a second simulation, we choose $\mu_1 = -2,$ $\mu_2 = -12,$ $\sigma_1 = 1,$ and $\sigma_2 = 1,$ i.e., we consider a jump in the mean of $N_1$ versus $N_2$ of size $n^{-1/2}\cdot (-10) = -0.1.$ Again, we have $n\Delta^2 = 100,$ so that we might hope to detect the jump in the mean. In Figure \ref{RCP:fig:jumpMean}, we depict one realization of $X^{(n)}$ under the alternative and the empirical values of the test statistic $\mathcal{S}^{1/2}_n$ under the null and under the alternative hypothesis. Even though we can see the jump in the expected value after $n \approx 3000$ time steps in the realization of $X^{(n)}$ very clearly, the empirical distributions of the test statistic $\mathcal{S}^{1/2}_n$ suggest that the null hypothesis ``no change point'' is harder to distinguish from the alternative hypothesis ``there exists one change point'' compared to our first simulation in Figure \ref{RCP:fig:jumpVol}. Moreover, if we interpret the observations $\mathcal{Y}^{(n)}$ as the discretely observed scaled increments of a diffusion process $X$, we see from a comparison of these two simulation studies that the detection of a jump in drift component of $X$ is harder than in its volatility component. In more detail, in order to guarantee that the condition in \eqref{RCP:ass:sizeCP} holds true and hence we are able to distinguish between the null and alternative, a change in the drift component has to converge to infinity, while the change in the volatility component might even go to zero as $n\rightarrow \infty.$ Note that these observations are also consistent with the theoretical results in \cite{AHHR09, JWY18}.

\begin{figure}[H]
    \centering
    \small{\textit{Change point model with parameters $\mu_1 = -2$ and $\mu_2 = -12$}}\\
    \includegraphics[scale = 0.43]{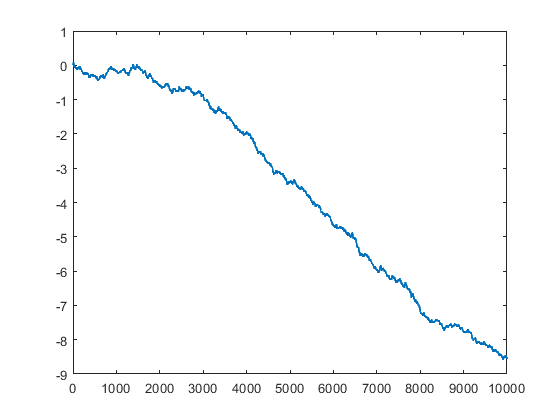} \quad
     \includegraphics[scale = 0.43]{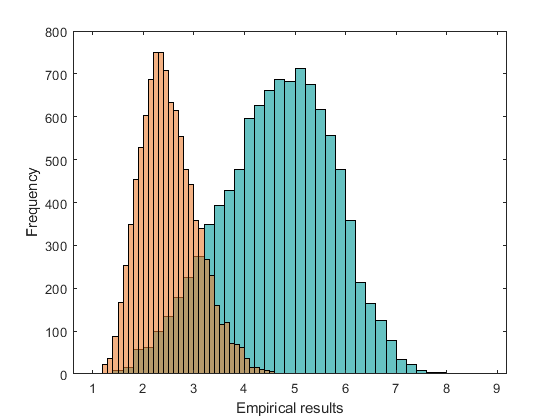}
    \caption{Left: One path of $X^{(n)}$ under the alternative with change point at $n = 3024$ time steps. Right: Empirical values of $\mathcal{S}^{1/2}_n$ under the null (orange) and under the alternative (turquoise).}
    \label{RCP:fig:jumpMean}
\end{figure}

For both simulations, we also calculate the empirical distribution of $\delta^2(\hat{k}_n - k^*_n).$ In both cases, we obtain that the empirical distribution replicates the theoretical result from Theorem \ref{RCP:res:limitDist}, cf. Figure \ref{RCP:fig:limitDist}, where we depict the empirical distribution of $\argmax_{u\in (-\infty, \infty)} \hat{W}(u)$ versus the empirical distribution of $\sigma_A^2 \delta^2(\hat{k}_n - k^*_n)$ for a change point in the mean.

\begin{figure}[H]
    \centering
    \includegraphics[scale = 0.43]{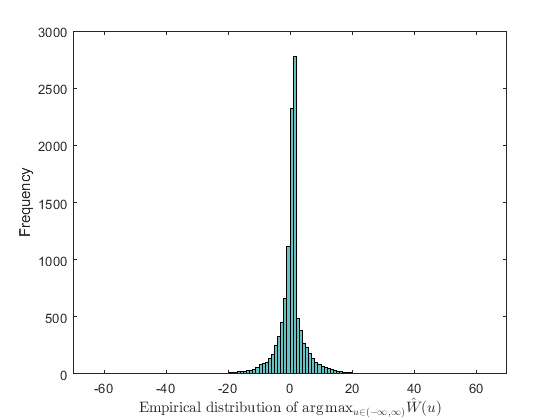}\quad
    \includegraphics[scale = 0.43]{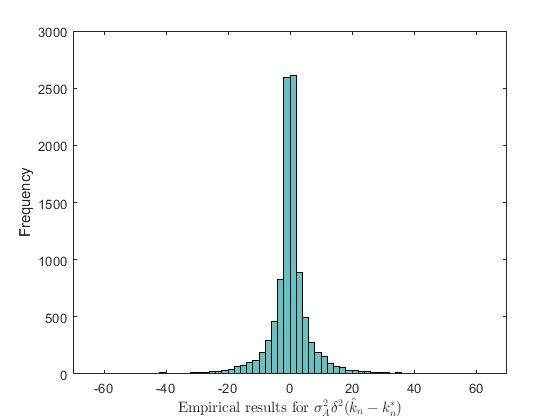}
    \caption{Left: Empirical distributions of $\argmax_{u\in (-\infty, \infty)} \hat{W}(u).$ Right: Empirical distribution of the deviation $\sigma^2_A\delta^2(\hat{k}_n - k^*_n)$.}
    \label{RCP:fig:limitDist}
\end{figure}

\textbf{Parametric change point detection for weakly dependent observations:} Although, we developed our theory for independent observations, weak dependencies between subsequent observations do not ruin our empirical results. To see that, for $a \in (-1,1)$, let 
\[\tilde{Y}^{(n)}_1 := Y^{(n)}_1, \quad \tilde{Y}^{(n)}_k := a \,Y^{(n)}_{k-1} + \sqrt{1-a^2}\, Y^{(n)}_{k} \quad \text{ for } k \geq 2,\]
and $\tilde{X}^{(n)}(t) := \sum_{k=1}^n\tilde{X}^{(n)}_k \1_{\{nt \in [k,k+1)\}}$ with $\tilde{X}^{(n)}_k := n^{-1/2}\sum_{j=1}^k\tilde{Y}^{(n)}_j.$ In the following simulation, we choose again $\mu_1 = \mu_2 = -2,$ $\sigma_1 = 1,$ and $\sigma_2 = 1.1.$ Moreover, we choose $a = 1/2.$ In Figure \ref{RCP:fig:jumpVol-Dependent}, we depict one realization of $\tilde{X}^{(n)}$ and the empirical values of the test statistic $\mathcal{S}^{1/2}_n$ under the null and alternative hypothesis. We observe that even for weakly dependent observations, the test statistic $\mathcal{S}^{1/2}_n$ is still able to distinguish between the null and alternative hypothesis. Moreover, also the empirical distribution of $\delta^2(\hat{k}_n - k^*_n)$ replicates the theoretical result from Theorem \ref{RCP:res:limitDist} (cf. Figure \ref{RCP:fig:limitDist-Dependent}).

\begin{figure}[H]
    \centering
    \small{\textit{Change point model for weakly dependent observations}}\\
    \includegraphics[scale = 0.43]{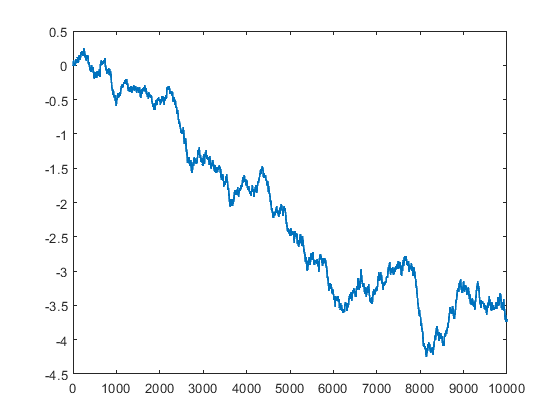}\quad 
    \includegraphics[scale = 0.43]{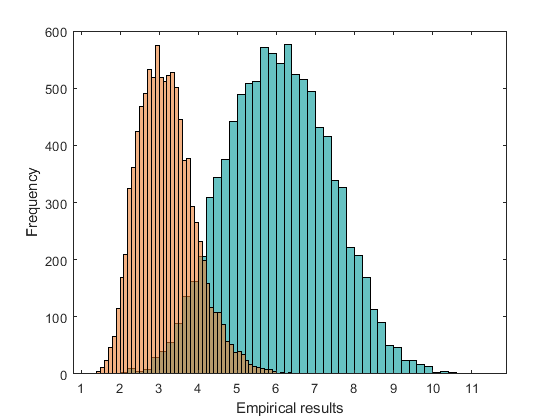}
    \caption{Left: One path of $\tilde{X}^{(n)}$ under the alternative with change point at $n = 2522$ time steps. Right: Empirical values of $\mathcal{S}^{1/2}_n$ under the null (orange) and under the alternative (turquoise).}
    \label{RCP:fig:jumpVol-Dependent}
\end{figure}

\begin{figure}[H]
    \centering
    \includegraphics[scale = 0.43]{images/disthatW2.png}\quad
    \includegraphics[scale = 0.43]{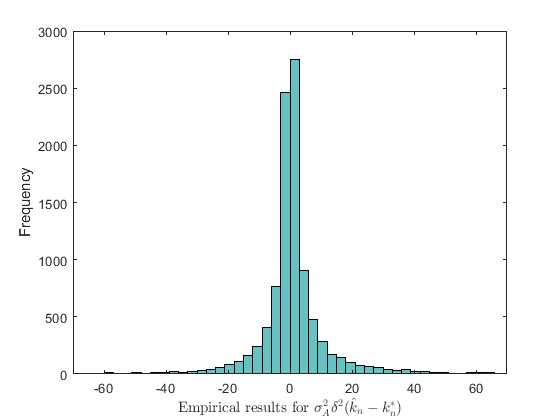}
    \caption{Left: Empirical distribution of $\argmax_{u\in (-\infty, \infty)} \hat{W}(u).$ Right: Empirical distribution of the deviation $\sigma^2_A\delta^2(\hat{k}_n - k^*_n)$.}
    \label{RCP:fig:limitDist-Dependent}
\end{figure}

Our empirical results therefore suggest that at least for weakly dependent observations, we are probably able to establish the stated asymptotic properties for the test statistic $\mathcal{S}_n$ and the estimators $\hat{k}_n$ and $\hat{\lambda}_n.$ \newline

\textbf{Non-parametric change point detection in the volatility process:} For many practical applications, an approximation as in \eqref{RCP:eq:limitDiff} does not describe the underlying structure of the observations well. In contrast, one might be interested in whether or not there is a jump in the volatility process $\sigma: \Omega \times [0,T] \rightarrow \R_+$ of an It\^o-semimartingale. The authors in \cite{BJV17} developed a statistical change point theory to detect, among others, a ``local jump'' in the  volatility process such that $|\sigma^2(\lambda) - \lim_{s\uparrow \lambda} \sigma^2(s)| > 0$ for some $\lambda \in (0,1).$ Let us consider a volatility process of the form
\[\sigma(t) = \left(\int_0^t c \cdot \rho \,dW(s) + \int^t_0 \sqrt{1-\rho^2} \cdot c \,dW^{\perp}(s) + 1\right) \cdot v(t)\]
which fluctuates around a deterministic seasonality function
\[v(t) = 1 - 0.2\sin\left(\frac{3}{4}\pi t\right), \quad t\in [0,1],\]
with $c = 0.1$ and $\rho = 0.5,$ where $W^{\perp}$ is a standard Brownian motion independent of $W.$ Note that the authors of \cite{BJV17} studied the same volatility process but for a deterministic location of the change point. Again, we simulate $\lambda^*_n$ by \eqref{RCP:eq:stopTime} and add one jump of size $0.3$ at time $\lambda^*_n$ to $\sigma.$ Since the volatility process is time-dependent, we apply the test statistic $V^*_{n,u_n}$ introduced in \cite{BJV17} instead of $\mathcal{S}_n$. Let $\Delta X^{(n)}_k := X^{(n)}_{k}- X^{(n)}_{k-1},$ $k\geq 1,$ be the increments of an It\^o-semimartingale $X$ with the volatility process above and constant drift equal to $-2$, recorded at discrete time steps $t^{(n)}_k,$ $k\geq 1$. Then, a reasonable test statistic is
\[V_{n,u_n}^* := \max_{i=k_n, \cdots, n-k_n} \left|\frac{\frac{n}{k_n}\sum_{j=i-k_n+1}^i (\Delta X^{(n)}_j)^2 \1_{\{|\Delta X^{(n)}_j| \leq u_n\}}}{\frac{n}{k_n}\sum_{j=i+1}^{i+k_n} (\Delta X^{(n)}_j)^2 \1_{\{|\Delta X^{(n)}_j|\leq u_n\}}} - 1\right|\]
where $k_n \rightarrow \infty.$ The core idea of the test statistic is to utilize a local two-sample $t$-test over $k_n$ asymptotically small blocks and take all overlapping blocks of $k_n$ increments into account. Moreover, we truncate the increments of $X$ by $u_n$ to exclude large squared increments which are ascribed to jumps. In \cite{BJV17}, the authors suggest to take $u_n = \sqrt{2\log(n)}n^{-1/2}$ and $k_n = C(\log(n))^{1/2} n^{1/2},$ for some $C > 1.$\par 
In Figure \ref{RCP:fig:jumpVolProcess}, we depict one realization of $\sigma$ under the null and under alternative hypothesis. Moreover, in Figure \ref{RCP:fig:jumpVolProcessTests}, we depict one realization of $X$ under the alternative and the empirical values of the test statistic \[\mathcal{V}_n:=\sqrt{\frac{\log(m_n)k_n}{2}}V^*_{n,u_n} - 2\log(m_n) - \frac{1}{2}\log\log(m_n) - \log(3)\]
under the null and under the alternative hypothesis. Here, $m_n := \lfloor n/k_n\rfloor.$ Accoring to Proposition 3.5 in \cite{BJV17}, the test statistic $\mathcal{V}_n$ converges in distribution under the null hypothesis to a Gumbel distribution. We observe that the test statistic in \cite{BJV17} can fairly good distinguish between the null and alternative even if the location of the change point has been sampled from the distribution in \eqref{RCP:eq:stopTime}.

\begin{figure}[H]
    \centering
    \includegraphics[scale = 0.43]{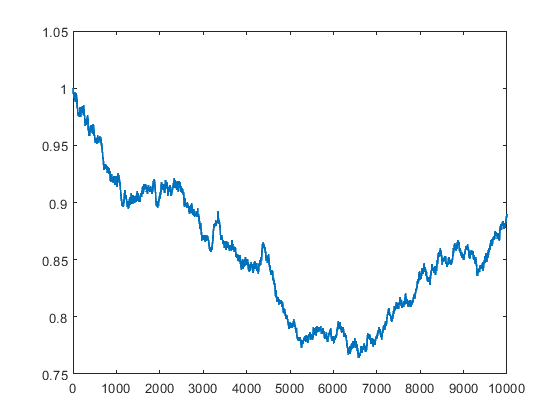} \quad
    \includegraphics[scale = 0.43]{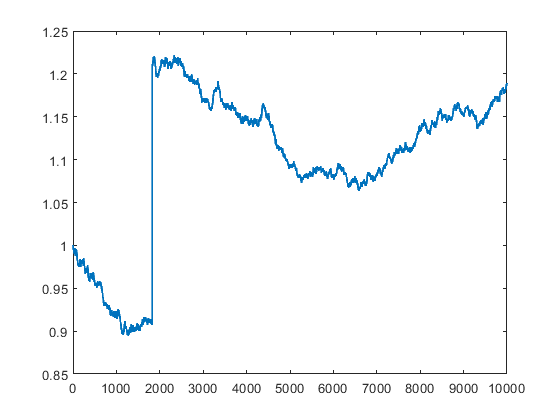}
    \caption{One realization of the volatility process $\sigma$ under the null (left) and under the alternative (right).}
    \label{RCP:fig:jumpVolProcess}
\end{figure}

\begin{figure}[H]
    \centering
    \textit{Non-parametric change point model with a jump in the volatility process of size $0.3$}\\
    \includegraphics[scale=0.43]{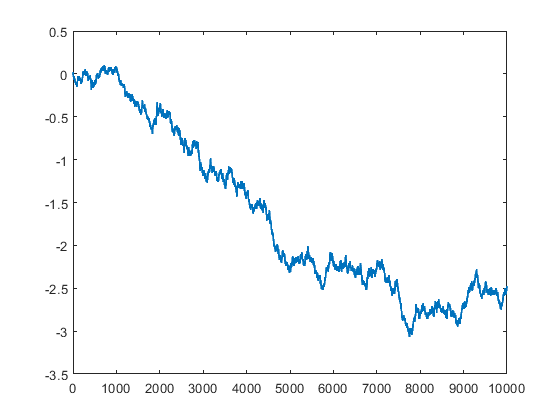}\quad
    \includegraphics[scale=0.43]{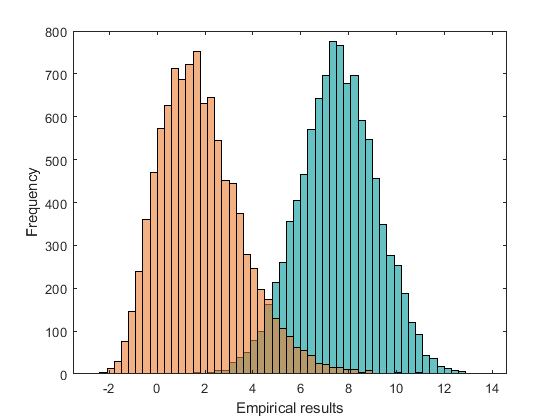}
    \caption{Left: One path of $X$ under the alternative with change point at $n = 1830$ time steps. Right: Empirical values of $\mathcal{V}_n$ under the null (orange) and under the alternative (turquoise).}
    \label{RCP:fig:jumpVolProcessTests}
\end{figure}

\textbf{Conclusion:} The starting point of our work was to generalize the theory in Csörg\H{o} and Horváth \cite{CH97} to randomly occurring change points in the model parameters, where, in particular, the location of the change point is allowed to depend on the data itself. In our simulation study, we generated the location of the change point from the stopping time in \eqref{RCP:eq:stopTime}. This stopping time is a rather simple way to choose the location of the change point depending on the data. From a financial point of view it is still quite interesting: the process X in \eqref{RCP:eq:limitDiff} might be an approximation for log prices of a financial asset containing a change point. Then, the stopping time in \eqref{RCP:eq:stopTime} causes the change in the model parameters of the log price if the price drops below some critical value $\kappa \in \R.$ It also shows that our theory is flexible enough to be applied to even more complex dependence relationships between the location of the change point and the observed data. \par 
Finally, our simulations for the case of weakly dependent observations as well as the non-parametric case suggest that change point theory in these settings still work even if the location of the change point depends on the data.

\section*{Acknowledgement}
Financial support by MATH+ through project funding AA4-4 “Stochastic modeling of intraday electricity
markets” is gratefully acknowledged.

\bibliographystyle{plain}
\bibliography{lit}

\end{document}